\documentclass[]{article}
\usepackage{appendix}
\usepackage{rotating}
\usepackage{graphicx} 
\usepackage{longtable}
\usepackage{diagbox}
\usepackage{amsmath, amssymb, amsthm} 

\usepackage{xcolor,color}

\usepackage{caption, subcaption} 
\usepackage{xcolor} 
\usepackage[left=2.5cm, right=2.5cm, top=2cm, bottom=3cm]{geometry} 
\usepackage{romanbar}
\usepackage{algorithm, algorithmic} 
\usepackage{array} 
\usepackage{booktabs} 
\usepackage{hyperref} 
\newcommand{\email}[1]{\texttt{#1}}

\newcommand{\norm}[1]{\left\Vert {#1}\right\Vert}
\newcommand{\seminorm}[1]{\left\vert {#1}\right\vert}
\newcommand{\innerprocuct}[1]{\left( {#1}\right)}
\newcommand{\M}[1]{\mathcal{M}\left( {#1}\right)}
\newcommand{\order}[1]{\mathcal{O} ({#1})}

\newcommand{\B}[1]{B\left( {#1}\right)}

\newcommand{\GM}[1]{\mathcal{G}_{\mathcal{M}}\left( {#1}\right)}

\newcommand{\vx}{\mathbf{x}}

\newcommand{\vy}{\mathbf{y}}

\newcommand{\vu}{\mathbf{u}}

\newcommand{\vv}{\mathbf{v}}
\newcommand{\vm}{\mathbf{m}}
\newcommand{\vw}{\mathbf{w}}

\newcommand{\vr}{\mathbf{r}}

\newcommand{\vM}{\mathbf{M}}

\newcommand{\bz}{\boldsymbol{0}}
\newcommand{\Ad}{\mathcal{A}} 
\newcommand{\I}{\normalfont{\Romanbar{1}}} 
 
\newcommand{\shuo}[1]{{\color{black}#1}}

\newtheorem{thm}{Theorem}[section]

\newtheorem{cor}[thm]{Corollary}
\newtheorem{remark}[thm]{Remark}
\newtheorem{example}{Example}[section]

\newtheorem{assumption}{Assumption}[section]
\newtheorem{lemma}[thm]{Lemma}
\newtheorem{proposition}[thm]{Proposition}

\title{BDF schemes for accelerated gradient flows in projection-free approximation of non-convex constrained variational minimization}
\author{Guozhi Dong\thanks{School of Mathematics and Statistics, Hunan Research Center of the Basic Discipline for Analytical Mathematics, HNP-LAMA, Central South University, Changsha 410083, China 
  (\email{guozhi.dong@csu.edu.cn}).}
\and Zikang Gong\thanks{MOE-LCSM, School of Mathematics and Statistics, Institute of Interdisciplinary Studies, Hunan Normal University, Changsha 410081, China  
  (\email{201830124045@hunnu.edu.cn}, \email{ziqingxie@hunnu.edu.cn}).
  }
 \and Ziqing Xie\footnotemark[2]
\and Shuo Yang\thanks{Beijing Institute of Mathematical Sciences and Applications, Beijing, 101408, China
  (\email{shuoyang@bimsa.cn}).}
  }

\begin{document}
\date{}
\maketitle

\renewcommand{\thefootnote}{\relax} 
\renewcommand{\thefootnote}{\arabic{footnote}} 

\begin{abstract}
\shuo{

We propose novel algorithms combining accelerated gradient flows with linearized projection-free treatments of non-convex constraints and BDF pseudo-temporal discretization for quadratic energy minimization. 
A general framework is developed to analyze constraint violations in such projection-free techniques for quadratic constraints. This analysis proves to be universal to all projection-free iterative methods, and constraint error bounds depend solely on iterate regularity.
For BDF-$k$ ($k\in\{1,2,3,4\}$), we derive both unconditional and conditional high-order constraint violation estimates for accelerated gradient flows using our framework. We further discover a new family of BDF-$k$ accelerated gradient methods achieving modified energy stability for arbitrary $k\in\mathbb{Z}^+$.
Numerical experiments validate our theoretical results and demonstrate superior efficiency and accuracy compared to existing gradient flow approaches.
}
\vspace{0.5cm}

\noindent{\it Keywords\/}: Accelerated gradient flow, quadratic constraint, projection free, constraint violation estimate, backward differentiation formula (BDF), energy stability

\vspace{0.5cm}
\noindent\textit{2020 Mathematics Subject Classification}: 65K10, 65N12.
\end{abstract}

\section{Introduction}
\subsection{Context and background}\label{sec:background}

Many mathematical models in science and engineering involve minimizing (quadratic) energies under non-convex constraints. Examples include the vector models (Oseen-Frank and Ericksen) for nematic liquid crystals \cite{frank1958liquid,ericksen1991liquid,de1993physics}, large deformations of nonlinear plates \cite{friesecke2002theorem,bhattacharya2016plates}, and ground states of Bose-Einstein condensates \cite{lieb2001bosons,dalfovo1999theory}, {among others}. The development of efficient numerical methods and/or simulation for these challenging problems has attracted significant research interest in recent years. In infinite-dimensional spaces setting, such problems involve topics in analysis of partial differential equations (PDEs), numerical methods for PDEs, calculus of variations, optimization, and scientific computing, which are key areas in applied and computational mathematics.


In general, such models can be expressed as \shuo{constrained minimization problems of the following form:
\begin{equation}\label{eq:cons_min}
    \underset{\vu\in \Ad}{\operatorname{minimize}}\; E[\vu], \quad E[\vu]:=\frac{1}{2}\mathcal{M}(\vu,\vu),
\end{equation}
where $\mathcal{M}:U\times U\to\mathbb{R}$ is a generic symmetric bilinear form on a Hilbert space $U$ defined on the bounded Lipschitz domain $\Omega\subset\mathbb{R}^d$. The non-convex admissible set $\Ad\subset U$ is defined as 
\begin{equation}\label{eq:admissible_set}
 \Ad:=\{\vu\in U_D: B(\vu,\vu)=\vu_c \; \text{ a.e. in }\Omega\},
\end{equation}
where $B:U\times U\to Z$ is a generic symmetric bilinear operator into another function space Z, given some $\vu_c\in Z$. Furthermore, $U_D$ represent a subset of $U$ where Dirichlet boundary conditions are imposed for $\vu$ and/or derivatives of $\vu$ on a part of the boundary $\Gamma_D\subset\partial\Omega$.  
}  

Gradient descent type algorithms are widely employed for problems of the form \eqref{eq:cons_min}, and the treatment of the constraint in $\Ad$ is crucial for enhancing the effectiveness, efficiency, and accuracy of these numerical methods.
A popular approach involves linearizing the constraint at each iteration of the gradient flow by solving for increments within a subspace $\mathcal{F}_{\vu^n} \subset U$, \shuo{with the definition  
\begin{equation}\label{eq:tangent_F}
 \mathcal{F}_{\vu}:=\{\vv\in U_0: B(\vu,\vv)=\bz \; \text{ a.e. in }\Omega\},
\end{equation}
where $\vu^n$ denotes the previous iterate and $U_0$ represents the subspace where functions are vanishing on $\Gamma_D\subset\partial\Omega$. 
More specifically, the gradient flow entails computing $d_t\vu^{n+1}\in\mathcal{F}_{\vu^n}$ with given $\vu^{n}\in U$ and a time step $s>0$ such that 
\begin{equation}\label{eq:disc-proj-free-gf}
\innerprocuct{d_t\vu^{n+1},\vv}_{U^*,U}+\delta E[\vu^{n}+sd_t\vu^{n+1}](\vv) =0,
\end{equation}
for all $\vv\in\mathcal{F}_{\vu^n}$ and updating $\vu^{n+1}=\vu^{n}+sd_t\vu^{n+1}$. 
Here $\delta E[\cdot]:U\to U^{*}$ denotes the first variation of the functional $E$ and $\delta E[\vu](\vv)=\mathcal{M}(\vu,\vv)$.} 
In specific cases, an additional projection step onto $\Ad$ can be performed. For example, \cite{alouges1997new} proposes a stable projection method for the Oseen-Frank model. With the spherical constraint $|\vu|^2=1$, a natural projection step defined as $\vu^n_{proj} := \frac{\vu^n}{|\vu^n|}$ can be employed after each iteration $\vu^n$. Additionally, in the case of an $L^2$ normalizing constraint, techniques combining projection with Lagrange multipliers for normalizing gradient methods are presented in \cite{LiuCai21}, aiming to make the entire trajectory sitting on the spherical manifold. For broader constraints, a coupled primal-dual accelerated flow based on augmented Lagrange multipliers is discussed in \cite{CheDonLiuXie23}.


However, a stable projection step is not always feasible. For instance in the large bending deformations of nonlinear plates, the matrix-valued isometry constraint $\nabla\vu^T\nabla\vu=I_2$ poses challenges in designing a stable projection step, to the best of our knowledge. Even for vector models of nematic liquid crystals with the spherical constraint, ensuring stability in the projection approach may require spatial discretization that satisfies specific restrictive conditions to maintain certain monotonicity properties \cite{bartels2005stability,nochetto2017finite}.

Taking into account these limitations, the projection-free method, \shuo{namely solving \eqref{eq:disc-proj-free-gf} for $d_t\vu^{n+1}\in\mathcal{F}_{\vu^n}$ without an additional projection step,} becomes a proper alternative in more general scenarios, and it has been investigated for various models, such as \cite{bartels2013approximation,bartels2016projection,nochetto2022gamma,bonito2023numerical,bartels2017bilayer,bartels2020stable,bonito2023gamma,hrkac2019convergent,feischl2017eddy}.
Introducing a pseudo time $t$, a continuous \shuo{counterpart of \eqref{eq:disc-proj-free-gf} reads as follows: $\dot{\vu}$ belongs to $\mathcal{F}_{\vu}$ and } 
\begin{equation}\label{eq:cont-proj-free-gf}
\innerprocuct{\dot{\vu},\vv}_{U}+\delta E[\vu](\vv) =0,
\end{equation}
for all $\vv\in\mathcal{F}_{\vu}$. Additionally, \eqref{eq:disc-proj-free-gf} can be viewed as a semi-implicit temporal discretization of \eqref{eq:cont-proj-free-gf} by approximating $\dot{\vu}$ via a BDF-1 scheme.

In such projection-free approaches, the constraint cannot be satisfied exactly. However, the constraint violation is controlled linearly by $s$ independent of the iteration number \shuo{$n$ as 
\begin{equation}\label{eq:linear-control-cons}
\|B(\vu^n,\vu^n)-\vu_c\|_Z\lesssim s, 
\end{equation}
for any $n\in\mathbb{N}$. The constraint violation formally converges to $0$ in the limit $s\to0$.
Remarkably, in the projection-free methods, the constraint violation is closely tied to optimization accuracy. In practice, choosing a small $s$ may yield a reasonably accurate solution, but this typically increases the number of total iterations with higher computational costs. Therefore, a trade-off arises between accuracy and efficiency when selecting $s$. Consequently, developing projection-free methods with higher-order constraint consistency is highly desirable.    
}
   
Recently, Akrivis, Bartels and Palus \cite{bartelsakrivis2023quadratic} investigate a variant of projection-free gradient flow by approximating $\dot{\vu}$ via a BDF-2 scheme and also modify the tangent space to be $\mathcal{F}_{\hat{\vu}^n}$ at $\hat{\vu}^n:=2\vu^{n-1}-\vu^{n-2}$. They prove an unconditional $\order{s}$ estimate for constraint violation, and an $\order{s^2}$ estimate under a discrete regularity assumption
\begin{equation}\label{eq:disc-reg-assump-ABP}
s^2\sum_{n=2}^N\|d^2_t\vu^n\|_U^2\le C,
\end{equation}
with a uniform constant $C>0$.

Despite its simplicity and stability, the gradient descent type method may converge slowly in the presence of no strong convexity or significant anisotropy. To address this limitation, accelerated gradient flows in Sobolev spaces have gained much attention, fueled by progress in second-order dynamics for convex optimization and their applications in variational problems. 
This concept traces back to foundational works in unconstrained convex optimization by Polyak \cite{Pol64}, Nesterov \cite{Nesterov1983}, Attouch et al. \cite{AttGouRed2000}, and Su et al. \cite{SuBoyCan2016JMLR}. Recent research highlights advantages in efficiency and robustness of these dynamics over traditional gradient flows in various variational contexts, as noted in \cite{CheDonLiuXie23, DonHinZha2021SIIMS, calder2019pde,bao2024convergence}.

In particular, for models of nonlinear plates characterized by non-convex constrained minimization problems in the form \eqref{eq:cons_min}, Nesterov-type and heavy-ball acceleration techniques are integrated with the projection-free strategy as discussed in \cite{DonGuoYan24}. The result there demonstrates that projection-free accelerated gradient flows significantly reduce the required number of iterations compared to conventional gradient flows for complex models, such as bilayer or prestrained plates with substantial anisotropy (the methods have comparable speed for simpler models). Furthermore, using the BDF-1 scheme, it is proved that constraint violation remains $\mathcal{O}(s)$, while the BDF-2 scheme achieves a convergence rate of $\order{s^3}$ for constraint violations as a computational observation. However, there is a lack of analysis regarding the BDF-2 scheme for accelerated gradient flows, not to mention systematic investigations into higher-order BDF schemes.

Inspired by technique tools in \cite{bartelsakrivis2023quadratic} and computational phenomenon in \cite{DonGuoYan24}, we aim to \shuo{investigate constraint violations for general BDF-$k$ projection-free iterative methods. Specifically, we design and analyze accelerated gradient flows with BDF-1/2/3/4 schemes in a general set-up in this work.} 
We write a continuous version of projection-free accelerated gradient flow for \eqref{eq:cons_min} as follows: $\dot{\vu}$ belongs to the tangent space $\mathcal{F}_{\vu}$ and 
\begin{equation}\label{eq:second order flow problem}
\innerprocuct{\ddot{\vu},\vv}_{U}+\frac{\alpha}{t}\innerprocuct{\dot{\vu},\vv}_{U}+\delta E[\vu](\vv) =0,
\end{equation}
for all $\vv\in\mathcal{F}_{\vu}$ \shuo{and $\dot{\vu}(0)=\bz$}. The damping coefficient is $\alpha/t$. 
\shuo{Following established convex optimization theory (see e.g. \cite{AttChbPeyRed2018}), we adopt the parameter range $\alpha \geq 3$ in this work, although we face non-convex constrained minimization problems \eqref{eq:cons_min}. For comprehensive discussion of the choice of $\alpha$ and the damping term, we refer to \cite{AttChbPeyRed2018,AttBotCse2023JEMS,BotDonElbSch2022FCM}.}  
It is worth noting that \eqref{eq:second order flow problem} is formally a weak formulation of dissipative hyperbolic type PDEs analogous to the second-order dynamics (ODE) proposed by Su et al. \cite{SuBoyCan2016JMLR}.

Backward differentiation formulas (BDFs) are frequently used for the temporal discretization of evolutionary equations, providing derivative approximations at the current time-step with varying orders of accuracy based on historical function values. 
Significant contributions on high-order BDFs for parabolic equations can be found in, e.g., both quasi-linear and nonlinear cases \cite{Buyang-IMPLICIT-EXPLICIT-BDF-METHODS-FOR-A-CLASS-OF-PARABOLIC-EQUATIONS-WITH-NON-SELFADJOINT-OPERATORS,BuyangAllen–CahnEquation,Akrivis-BDF-quasilinearParabolicEquations,Akrivis-Lubich-BDF-quasi-linear-parabolicEquations,Akrivis-Stability-BDF-Nonlinear-ParabolicEquations}, as well as applications to parabolic equations on evolving surfaces \cite{Lubich-BDF-parabolicPDE-on-EvolvingSurfaces,Lubich-Linearly-implicit-FullDiscretization-of-SurfaceEvolution} and fractional evolutionary equations \cite{BuyangFractionalEvolutionEquations}.
 Moreover, in \cite{LLG} the Landau–Lifshitz–Gilbert (LLG) equation is examined using implicit BDF time discretizations of orders 1 to 5, where an approximate tangent space enforcing the orthogonality constraint is employed. For BDF methods of orders 3 to 5, quasi-optimal error estimates
are proved, while these results hinge on the existence of sufficiently regular solutions to the LLG equation. 


\subsection{Contribution and structure of the paper}
\shuo{
In this work, we propose and analyze accelerated gradient flow methods for approximating solutions to the non-convex constrained minimization problem \eqref{eq:cons_min} within an abstract setting. We formulate a general framework to estimate constraint violations in projection-free schemes, where linearized constraints are enforced via BDF approximations.
Novel projection-free accelerated gradient flow schemes are developed, and higher-order constraint violations of such accelerated schemes are shown, compared to existing projection-free gradient flows. 
In addition, modified energy stability and weak convergence are established for these accelerated flows. 
We emphasize that the computational acceleration effect of these flows has been discussed thoroughly in \cite{DonGuoYan24}, and our focus here is on analyzing constraint violations, which is critical for approximation accuracy, as indicated in Section~\ref{sec:background}.
The results and structure of this paper are summarized as follows.
}

\shuo{Section \ref{sec:assump-ex} states assumptions on operators and spaces for \eqref{eq:cons_min} and provides examples. Section \ref{sec:diff} introduces BDF formulas and useful results about linear difference equations.}

\shuo{
In Section \ref{sec:cons-vio-bdfn}, we develop a general framework to estimate the constraint violation 
$\|B(\vu^n,\vu^n)-\vu_c\|_Z$ in arbitrary projection-free schemes, where the linearized constraint,  
\begin{equation}\label{eq:linear-cons-0}
B(\dot{\vu}^n,\hat{\vu}^n)=0,
\end{equation}
is imposed at each iteration using the BDF-$k$ approximation $\dot{\vu}^n$ of time derivatives and the BDF-$k$ extrapolation $\hat{\vu}^n$.    
The analysis proceeds in two steps:
(1) In Section \ref{sec:roadmap} we state the crucial algebraic identity \eqref{eq:abstract-identity} relating a sum of quadratic terms $B(\vu^n,\vu^n)$ to discrete temporal derivatives $d^j_t\vu^n$, and its detailed derivation is provided in Appendix~\ref{appendix:appendix-A}.
(2) In Section \ref{sec:cons-vio-general}, we employ \eqref{eq:abstract-identity} to obtain a general constraint violation expression in Theorem \ref{thm:bn-expression} and further an estimate in the form of 
\begin{equation}\label{eq:cons-upper-general-0}
\|B(\vu^n,\vu^n)-\vu_c\|_Z\lesssim \sum_{(j,n)\in\mathcal{J}}s^{2j}\|d^j_t \vu^{n}\|_U^2,
\end{equation}
where $\mathcal{J}\subset\mathbb{N}\times\mathbb{N}$ is an index set.
In particular, under the discrete regularity conditions
\begin{equation}\label{eq:discrete-regularity-general}
\sum_{(j,n) \in \mathcal{J}} s^{2j-q} \|d^j_t \vu^{n}\|_U^2 \leq C \text{ and/or } \max_{1\leq n\leq N}\|d^j_t \vu^{n}\|_U^2\leq C,
\end{equation}
with $q=3,3,4$ for BDF-$2,3,4$ respectively, we establish $\order{s^3}$ (BDF-2/3) and $\order{s^4}$ (BDF-4) estimates of constraint violations.
Notably, these results apply to any iterative scheme satisfying \eqref{eq:linear-cons-0}.
}

\shuo{
In Section \ref{sec:BDF1} we introduces a BDF-$1$ projection-free accelerated gradient flow scheme inspired by \eqref{eq:second order flow problem}. 
Then Section \ref{sec:bdf2-stability} presents the BDF-$2$ version, which enforces \eqref{eq:linear-cons-0} by computing $\dot{\vu}^n\in\mathcal{F}_{\hat{\vu}^n}$ at each iteration with the BDF-2 extrapolation $\hat{\vu}^n$.
The BDF-2 method satisfies the total energy stability
\begin{equation}\label{eq:GM-stab-0}
	\GM{\vu^n,\vu^{n-1}}+\norm{\dot{\vu}^n}^2_U\leq \GM{\vu^{n-1},\vu^{n-2}}+\norm{\dot{\vu}^{n-1}}^2_U,
\end{equation}
using a BDF-$2$ adapted bilinear form $\mathcal{G}_\mathcal{M}$, following the idea from \cite{bartelsakrivis2023quadratic} and extending the notion of $G$-stability from numerical ODEs. 
This stability ensures existence of a weak limit $\vu^*$ of $\{\vu^n\}_{n\in\mathbb{N}}$, and we characterize $\vu^*$ as a local minimizer along the limiting tangent direction.
Additionally, Section \ref{sec:cons-vio-BDF2} derives an unconditional $\order{s^2}$ constraint violation estimate by combining the stability \eqref{eq:GM-stab-0} and the framework presented in Section \ref{sec:cons-vio-bdfn}.
Our results improve upon the estimates for the BDF-2 projection-free gradient flows in \cite{bartelsakrivis2023quadratic}, gaining one order in both conditional and unconditional constraint violation estimates.
The key ingredient is that the trajectory of accelerated gradient flow is automatically endowed with stronger regularity in its second-order time derivative.
}

\shuo{
In Section \ref{sec:modified}, we further propose a new class of energy stable projection-free accelerated gradient methods with general BDF-$k$ approximations. The key modification replaces $\mathcal{M}(\vu^n,\vw)$ in \eqref{eq:discrete_flow_BDF2} with $\mathcal{M}(\widetilde{\vu}^n,\vw)$ in \eqref{eq:discrete_flow_BDFk_modified}, where $\widetilde{\vu}^n$ is defined in \eqref{eq:def-tilde-un}. 
The new scheme achieves stability with a modified total energy $E[\widetilde{\vu}^n]+\frac{1}{2}\|\dot{\vu}^n\|_U^2$, differing from the method in Section \ref{sec:BDF2}. The stability also ensures the existence of a weak limit $\vu^*$ and the unconditional $\order{s^2}$ constraint violation estimates for $2\leq k\leq 4$.
} 

Finally, we present numerical simulations in Section \ref{sec:numerical results}, focusing on an anisotropic Dirichlet energy and a model for prestrained plates. In the first example, we demonstrate that the constraint violation converges at rates of $\order{s}$, $\order{s^3}$, $\order{s^3}$, and $\order{s^4}$ for accelerated flows using BDF-1/2/3/4 schemes respectively. We also validate the satisfaction to the corresponding discrete regularity assumptions for BDF-2/3/4 methods and highlight the higher-order accuracy of our proposed methods compared to existing projection-free gradient flow methods. In the second example, we further explore the impact of spatial discretization refinement on the computational performance of the new algorithms, concluding that a sufficiently fine mesh or higher-order spatial discretization is necessary to maximize the potential of capabilities of the projection-free accelerated flows for more complicated problems like prestrained plates. 

\section{\shuo{Preliminaries}}\label{sec:setting}
\setcounter{equation}{0}
\subsection{\shuo{Assumptions and examples}}\label{sec:assump-ex}

\shuo{We first make the following assumptions on function spaces and operators appearing in the constrained minimization problem \eqref{eq:cons_min}.}
\shuo{
\begin{assumption}\label{assum:space-U}
Let $\Omega\subset\mathbb{R}^d$ be an open bounded Lipschitz domain. 
We assume the Hilbert space $U$ consists of functions $\vu:\Omega\to \mathbb{R}^m$ ($m\geq 1$) whose Jacobian $\nabla \vu:\Omega\to\mathbb{R}^{m\times d}$ exist in the weak sense. 
We assume that $U$ is embedded in $H^1(\Omega;\mathbb{R}^m)$, i.e.,
	\begin{equation}\label{assumption of embeddedness}
		\norm{\vu}_{H^1}\leq c_U\norm{\vu}_U,
	\end{equation}
where $c_U>0$ is a constant depending only on $U$.
\end{assumption}    
}

\shuo{ 
\begin{assumption}\label{assum:B}
For the generic bilinear operator $B:U\times U\to Z$ that describes the quadratic constraint, we assume its boundedness as
\begin{equation}\label{eq:B-upper}
 	\norm{\B{\vu,\vu}}_Z\leq c_Z\norm{\vu}_U^2,
 \end{equation}
 for all $\vu\in U$.
\end{assumption}     
}    

\begin{assumption}\label{assum:bilinear}
	The bilinear form $\mathcal{M}(\cdot,\cdot)$ satisfies the following properties:
	\begin{itemize}
    \item Coercivity: $\mathcal{M}(\vu,\vu)\ge C_1\|\vu\|_U^2$ for all $\vu\in U$.
    \item Continuity: $\mathcal{M}(\vu,\vv)\le C_2\|\vu\|_U\|\vv\|_U$ for each pair of $\vu,\vv\in U$.
  	\end{itemize}
\end{assumption}

Under Assumption \ref{assum:bilinear}, the direct method of the calculus of variations ensures the existence of the solutions of Problem \eqref{eq:cons_min} if the admissible set $\mathcal{A}$ is non-empty.

\shuo{
Using binominal formula $2a(a+b)=a^2+(a+b)^2-b^2$, we readily derive a useful identity for generic bilinear forms as
}
	\begin{equation}\label{binominal formula}
		2\mathcal{M} (d_t\vu^n,\vu^n)=\frac{1}{s}\left(\M{\vu^n,\vu^n}-\M{\vu^{n-1},\vu^{n-1}}\right)+s\M{d_t\vu^n,d_t\vu^n}.
	\end{equation}

Moreover, throughout the paper, we use $A_1:A_2$ to represent contraction between any pair of tensors $A_1,A_2$ with compatible orders and dimensions.

We next present a few examples to illustrate the general setup of the constrained minimization problem \eqref{eq:cons_min}. These examples will be utilized later in Section \ref{sec:numerical results} to validate the properties of the proposed algorithms. 
\begin{example}[anisotropic Dirichlet energy]\label{exampjle:anisotropic Dirichlet energy}
		We consider an anisotropic variant of the well-known Dirichlet energy, where the bilinear form $\mathcal{M}$ is defined as
        \begin{equation}\label{eq:dirichlet-energy}
		\mathcal{M}(\vu,\vu)=\innerprocuct{\nabla\vu,\nabla\vu \vM}_{L^2},
        \end{equation}
		where $\vu\in U:= H^1(\Omega;\mathbb{R}^m)$ and $\vM$ is a $d\times d$ diagonal positive definite matrix. 
        In this example, we consider $d=2$ or $3$ and $m=2$ or $3$ with $m\ge d$.
        When $\vM=I_d$ (identity matrix), the energy \eqref{eq:cons_min} with \eqref{eq:dirichlet-energy} reduces to the standard Dirichlet energy $E[\vu]=\frac12\|\nabla\vu\|_{L^2}^2$. 
        Moreover, we consider 
        \shuo{
        \begin{equation}\label{eq:dirichlet-energy-cons}
        B(\vu,\vv):=\vu\cdot\vv, \quad \vu_c := 1,
        \end{equation}
        and naturally we take $Z:=L^1(\Omega)$ in this case. 
        }
        \shuo{This is a point-wise unit length constraint for the vector field $\vu$.}
        \shuo{Furthermore, we impose Dirichlet boundary condition so that $\vu=\vu_D\text{ on }\Gamma_D\subset\partial\Omega$ with a given data $\vu_D\in H^1(\Omega;\mathbb{R}^m)$ satisfying $|\vu_D|=1$ a.e. in $\Omega$.} 
        \shuo{The Assumptions \ref{assum:space-U}, \ref{assum:B}, and \ref{assum:bilinear} are obviously satisfied.}
	\end{example}

\begin{example}[prestrained plates]\label{ex:plates}
We consider a simplified model of prestrained plates that entails the minimization of a bending energy subject to metric constraints, corresponding to the model studied in \cite{bonito2022ldg,bonito2023numerical} with the choice of Lam\'e coefficient $\mu=12$ and $\lambda=0$. We take $d=2$, $m=3$ and the bilinear form $\mathcal{M}$ reads 
\begin{equation}\label{eq:M-plates}
\mathcal{M}(\vu,\vv)=2(A^{-\frac{1}{2}}D^2\vu A^{-\frac{1}{2}},A^{-\frac{1}{2}}D^2\vv A^{-\frac{1}{2}})_{L^2(\Omega)},
\end{equation}
with $\vu,\vv\in U$ and $U:=H^2(\Omega;\mathbb{R}^3)$. 
We take 
\shuo{
\begin{equation}\label{eq:plates-cons}
 B(\vu,\vv):=\frac12(\nabla\vu^T\nabla\vv+\nabla\vv^T\nabla\vu),\quad \vu_c := A,
\end{equation}
and $Z:=L^1(\Omega;\mathbb{R}^{2\times2})$. 
}
\shuo{This is} a metric constraint with the target metric $A\in Z$ given as a symmetric positive definite matrix-valued function. 
In this example, we consider the Dirichlet boundary conditions 
\shuo{
\begin{equation}\label{eq:plates-BC}
 \vu = \phi,\quad \nabla\vu = \psi \quad \text{on }\Gamma_D,
\end{equation}
where} 
the boundary data $\phi\in [H^2(\Omega)]^3$ and $\psi \in [H^1(\Omega)]^{3\times2}$ are compatible with $A$, namely $\psi=\nabla\phi$ and $\psi^T\psi=A$ on $\Gamma_D$.
\shuo{The Assumptions \ref{assum:space-U}, \ref{assum:B}, and \ref{assum:bilinear} are obviously satisfied.}
\end{example}




\subsection{\shuo{Difference operators and difference equations}}\label{sec:diff}
We introduce the backward difference operator $d_t$, which approximates the first-order time derivative, as follows
 \begin{equation}\label{eq:BDF1-formula}
 d_t a_n=\frac{a_n-a_{n-1}}{s},
 \end{equation}
where {$a_n$ (with $n\in\mathbb{N}$) represents a function sequence and $s$ is an artificial time-step size}. Similarly, we define higher order discrete time derivative
as 
\begin{equation}\label{eq:Backward-formula-higher}
d_t^ka_n=\frac{d_t^{k-1} a_n-d_t^{k-1} a_{n-1}}{s},
 \end{equation}	
with $k\ge2$, $d^1_t:=d_t$ and $n\ge k$.
\shuo{
We readily express $d^k_ta_{n}$ for any $k\ge1$ and $n\ge k$ as  
\begin{equation}\label{eq:Backward-formula-expression}
    d^k_ta_{n}=s^{-k}\sum_{m=0}^k(-1)^m\binom{k}{m}a_{n-m}. 
\end{equation}
}

The general BDF-k approximation of time derivative and \shuo{$k$-th} extrapolation are defined as follows:
\begin{equation}\label{eq:bdf-formula-extrapolation}
\dot{\vu}^n:=\frac{1}{s}\sum_{j=0}^k \delta_j \vu^{n-j}, \quad \hat{\vu}^n:=\sum_{j=0}^{k-1} \gamma_j \vu^{n-j-1},
\end{equation}
where the coefficients $\{\delta_j\}_{j=0}^k$ and $\{\gamma_j\}_{j=0}^{k-1}$ \shuo{are defined as}
\shuo{
\begin{equation}\label{eq:def-bdf-k-coeff}
    \delta_0=\sum_{r=1}^k\frac{1}{r},\quad\delta_i=\frac{(-1)^i}{i}\binom{k}{i} \text{ for } i=1,\ldots,k,\quad \gamma_{j}=(-1)^{j}\binom{k}{j+1},
\end{equation}
where $j=0,\ldots,k-1$. The following properties of these coefficients and a combinatorial identity as follows will be useful in our analysis:
\begin{equation}\label{eq:useful-relation}
    \sum_{j=0}^k\delta_{j}=0,\quad\sum_{j=0}^{k-1}\gamma_{j}=1,\quad\sum_{j=0}^k(-1)^j\binom{k}{j}=0.
\end{equation}   
}
\shuo{We next define related coefficients $\{\tilde\delta_{j}\}_{j=0}^{k-1}$ that will be useful in our analysis.  
\begin{equation}\label{eq:def-tilde-delta}
\tilde\delta_0 := \delta_0,\quad \tilde\delta_{j} := \sum_{\ell=0}^{j}\delta_{\ell} \text{ for }1\leq j\leq k-1,
\end{equation}
and they satisfy that  
\begin{equation}\label{eq:sum-tilde-delta}
    \sum_{j=0}^{k-1} \tilde\delta_j=1.
\end{equation}  
}
\shuo{
We list values of these coefficients for BDF-$k$ schemes with $k=2,3,4$ as in Table \ref{tab:bdf-coefficients}. 
\begin{table}[ht]
\centering
\begin{tabular}{|*{10}{c|}} 
\hline
 & \multicolumn{3}{c|}{$\delta_j$} & \multicolumn{3}{c|}{$\tilde\delta_j$ }  & 
 \multicolumn{3}{c|}{ $\gamma_j$} 
 \\ \hline
j & $k=2$ & $k=3$ & $k=4$ &  $k=2$ &$k=3$ & $k=4$ &  $k=2$ &  $k=3$  &$k=4$ \\ \hline
$0$ & $3/2$ & $11/6$ & $25/12$ & $3/2$ & $11/6$ & $25/12$ & $2$ & $3$ & $4$ \\ \hline
$1$ & $-2$ & $-3$ & $-4$ & $-1/2$ & $-7/6$ & $-23/12$ & $-1$ & $-3$ & $-6$ \\ \hline
$2$ & $1/2$ & $3/2$ &  $3$ & \diagbox{}{} & $1/3$ & $13/12$ & \diagbox{}{} & $1$ & $4$ \\ \hline
$3$ & \diagbox{}{} & $-1/3$ & $-4/3$ & \diagbox{}{} & \diagbox{}{} & $-1/4$ & \diagbox{}{} & \diagbox{}{} & $-1$ \\ \hline
$4$ & \diagbox{}{} & \diagbox{}{} & $1/4$ & \diagbox{}{} & \diagbox{}{} & \diagbox{}{} & \diagbox{}{} & \diagbox{}{} & \diagbox{}{} \\ \hline
\end{tabular}
\caption{Coefficients for BDF-$k$ schemes with $k=2,3,4$.}
\label{tab:bdf-coefficients}
\end{table}
}

\shuo{
We next define the characteristic polynomial $\delta(z)$ with $z\in\mathbb{C}$ for the k-step BDF method as 
\begin{equation}\label{eq:delta-polynomial}
\delta(z):=\sum_{j=0}^k\delta_jz^j. 
\end{equation}
For $k\le 6$, $\delta(z)$ has only one unimodular root $1$ and all other roots lie outside the unit disc in the complex plane \cite{wanner1996solving}. 
Therefore, we factor out $1-z$ from $\delta(z)$ so that $\delta(z)=(1-z)\tilde\delta(z)$, with 
\begin{equation}\label{eq:tilde-delta-polynomial}
\tilde\delta(z):=\sum_{j=0}^{k-1}\tilde\delta_jz^j,
\end{equation}
where $\tilde\delta_j$ is defined as \eqref{eq:def-tilde-delta}, and all the roots of $\tilde\delta(z)$ have modulus bigger than $1$. 
Whence the rational function $1/\tilde\delta(z)$ is holomorphic in the open unit disk and admits a Taylor expansion for $|z|<1$ about the origin as
\begin{equation}\label{eq:taylor-expansion}
\frac{1}{\tilde\delta(z)}=\sum_{j=0}^{\infty}\eta_jz^j.
\end{equation}
Multiplying \eqref{eq:taylor-expansion} by $\tilde\delta(z)$ and comparing coefficients lead to linear homogeneous difference equations for $\eta_j$: 
\begin{equation}\label{eq:gamma-diff-eq}
\eta_0=\tilde\delta_0^{-1},\quad\sum_{j=0}^{k-1}\tilde\delta_j\eta_{n-j}=0,
\end{equation}
for all $n\geq 1$ with the convention that $\eta_j=0$ if $j<0$.
The solution can be expressed as  
\begin{equation}\label{eq:gamma-sol}
\eta_n=\sum_{i=1}^kC_ir_i^{-n},
\end{equation}
where $r_1,\ldots,r_k$ are $k$ roots of the characteristic polynomial \eqref{eq:tilde-delta-polynomial}, and constants $C_i$ are determined by initial values $\eta_0,\ldots,\eta_{k-1}$; see \cite{henrici1962discrete,wanner1996solving,gautschi2011numerical}. We note that these initial values can be computed recursively by \eqref{eq:gamma-diff-eq} with $n=0,1,\ldots,k-1$. Since $|r_i|>1$ for every $1\leq i\leq k$, it is clear that for all $N\ge0$ 
\begin{equation}\label{eq:sum_gamma_bdd}
\left|\sum_{n=0}^N\eta_n\right|\leq C,
\end{equation}
with a uniform constant $C$ independent of $N$. 
}

\shuo{
The solution of inhomogeneous difference equations with coefficients $\{\tilde\delta_{j}\}_{j=0}^{k-1}$ can be expressed in terms of the initial values, $\{\tilde\delta_{j}\}_{j=0}^{k-1}$ and $\{\eta_{j}\}_{j=0}^{\infty}$. We state the result in the next Lemma and refer to \cite{henrici1962discrete,gautschi2011numerical} for classical theory of difference equations.  
\begin{lemma}\label{lem:diff-eq-sol-general}
    For every fixed $2\le k\le 6$ and a given sequence $\{f_n\}_{n=0}^{\infty}$, suppose a sequence $\{a_n\}_{n=0}^{\infty}$ satisfy a linear inhomogeneous difference equation as 
    \begin{equation}\label{eq:diff-eq-general}
        \sum_{j=0}^{k-1} \tilde\delta_j a_{n-j}=f_n,\;\quad   \text{ for all  } n\geq k.
    \end{equation}
    Then the equation \eqref{eq:diff-eq-general} admits a solution 
    \begin{equation}\label{eq:diff-eq-sol-general}
    a_n=-\sum_{m=1}^{k-1}\sum_{\ell=k-m}^{k-1}\tilde\delta_{\ell}\eta_{n-\ell-m}a_m+\sum_{j=k}^n\eta_{n-j}f_j,\;\quad \text{ for all  } n\geq k.
    \end{equation}
\end{lemma} 
}
  
\shuo{
Expression \eqref{eq:diff-eq-sol-general} immediately yields an upper bound for $|a_n|$ in terms of the initial values $|a_1|,\ldots,|a_{k-1}|$ and $\sum_{j=k}^n|f_j|$. However, our subsequent analysis requires a sharper estimate, which we establish in the following lemma.
\begin{lemma}\label{lem:diff-eq-sol-sharper-estimate}
    For every fixed $2\le k\le 4$ and a sequence $\{a_n\}_{n=0}^{\infty}$ satisfying the linear inhomogeneous difference equation \eqref{eq:diff-eq-general}, the following estimate is valid: 
    \begin{equation}\label{eq:diff-eq-sol-sharper-estimate}
    \sum_{n=k}^N|a_n|^2\lesssim\sum_{n=1}^{k-1}|a_n|^2+\sum_{n=k}^N|f_n|^2,
    \end{equation}
    for all $N\ge k$ and the hidden constant depends on the BDF order $k$.  
\end{lemma} 
\begin{proof}
When $k=2$, the difference equation in the general form \eqref{eq:diff-eq-general} can be written as $2f_n=3a_n-a_{n-1}$, which implies that $|a_n|^2-\frac29|a_{n-1}|^2\leq \frac89|f_n|^2$ by the Cauchy-Schwarz inequality. Summing from $n=2$ to $N$ we observe that
 \begin{equation}\label{eq:anfn-diff-sum-1}
 |a_N|^2+\frac79\sum_{n=2}^{N-1}|a_n|^2-\frac{2}{9}|a_1|^2\leq\frac{8}{9}\sum_{n=2}^{N}|f_n|^2,
 \end{equation}
 which immediately leads to an inequality in the form of \eqref{eq:diff-eq-sol-sharper-estimate} for $k=2$. 

 When $k\ge3$, with $\beta:=-\tilde\delta_1/\delta_0$ and the equation \eqref{eq:diff-eq-general} we compute for $n\geq k+1$ 
 \begin{equation}\label{eq:fn-beta}
 f_n+\beta f_{n-1} = \delta_0a_n+\sum_{j=2}^{k-1}(\tilde\delta_j+\beta\tilde\delta_{j-1})a_{n-j}+\beta\tilde\delta_{k-1}a_{n-k}. 
 \end{equation}
Then by the Cauchy inequality, we derive
\begin{equation}\label{eq:an-sq-k-beta}
|a_n|^2\leq k\left(\sum_{j=2}^{k-1}\frac{(\tilde\delta_j+\beta\tilde\delta_{j-1})^2}{\delta_0^2}|a_{n-j}|^2+\frac{\beta^2\tilde\delta_{k-1}^2}{\delta_0^2}|a_{n-k}|^2+\left(\frac{1}{\delta_0}f_n+\frac{\beta}{\delta_0}f_{n-1}\right)^2\right). 
\end{equation}
Summing from $n=k+1$ to $N$, we derive  
\begin{equation}\label{eq:an-sq-sum-1}
\sum_{n=k+1}^N|a_n|^2\leq k\left(\sum_{j=2}^{k-1}\sum_{n=k+1-j}^{N-j}\frac{(\tilde\delta_j+\beta\tilde\delta_{j-1})^2}{\delta_0^2}|a_n|^2+\sum_{n=1}^{N-k}\frac{\beta^2\tilde\delta_{k-1}^2}{\delta_0^2}|a_{n}|^2\right)+C_k\sum_{n=k}^N|f_n|^2, 
\end{equation}
where $C_k$ is a constant depending on $k$. This further yields 
\begin{equation}\label{eq:an-sq-sum-2}
\left(1-k\left(\sum_{j=2}^{k-1}\frac{(\tilde\delta_j+\beta\tilde\delta_{j-1})^2}{\delta_0^2}+\frac{\beta^2\tilde\delta_{k-1}^2}{\delta_0^2}\right)\right)\sum_{n=k+1}^N|a_n|^2\lesssim \sum_{n=1}^{k}|a_n|^2+\sum_{n=k}^N|f_n|^2. 
\end{equation}
With the condition 
\begin{equation}\label{eq:an-sq-cond}
k\left(\sum_{j=2}^{k-1}\frac{(\tilde\delta_j+\beta\tilde\delta_{j-1})^2}{\delta_0^2}+\frac{\beta^2\tilde\delta_{k-1}^2}{\delta_0^2}\right)<1,
\end{equation}
the estimate \eqref{eq:an-sq-sum-2} immediately implies that 
\begin{equation}\label{eq:an-sq-sum-3}
\sum_{n=k+1}^N|a_n|^2\lesssim \sum_{n=1}^{k}|a_n|^2+\sum_{n=k}^N|f_n|^2. 
\end{equation}
Using \eqref{eq:diff-eq-sol-general} for $n=k$, we further derive that  
\begin{equation}\label{eq:an-sq-N=k}
|a_k|^2\lesssim \sum_{n=1}^{k-1}|a_n|^2+|f_k|^2. 
\end{equation}
We combine \eqref{eq:an-sq-sum-3} and \eqref{eq:an-sq-N=k} to obtain \eqref{eq:diff-eq-sol-sharper-estimate}, and the hidden constant depends on $k$. It remains to validate the condition \eqref{eq:an-sq-cond}. By a direct calculation with the values of $\tilde\delta_j$ given in Table \ref{tab:bdf-coefficients}, we verify that \eqref{eq:an-sq-cond} holds for $k=3,4$. 
\end{proof}
}

\shuo{
\begin{remark}
Notably, while \eqref{eq:an-sq-cond} fails for $k=5$, the analysis can be extended by rewriting $f_n + \beta_1 f_{n-1} + \beta_2 f_{n-2}$ (replacing the LHS of \eqref{eq:fn-beta}) to establish \eqref{eq:diff-eq-sol-sharper-estimate}. For conciseness, we restrict our analysis to $k~=~2,3,4$, omitting higher-order cases ($k \geq 5$).

It is worth noting that the validity of Lemmas \ref{lem:diff-eq-sol-general} and \ref{lem:diff-eq-sol-sharper-estimate} extends naturally to the case where $a_n$ and $f_n$ represent sequences of Sobolev functions, with the absolute values $|a_n|$, $|f_n|$ replaced by their corresponding function norms. This extension follows directly from the purely algebraic nature of the underlying arguments and norm-based estimates involved in both results. 
\end{remark}
}

\setcounter{equation}{0}
\section{Constraint violations for BDF-k schemes}\label{sec:cons-vio-bdfn}
\shuo{
In this section, we prove the constraint consistency orders for arbitrary projection-free iterative schemes. Notably, our analysis is independent of specific iterative scheme formulations, relying solely on the linearization \eqref{eq:linear-cons-0} of the quadratic constraint in \eqref{eq:admissible_set} through BDF approximations and extrapolations.    
}



\subsection{\shuo{Key algebraic identity}}\label{sec:roadmap}
The first step in estimating constraint violation involves deriving the crucial algebraic identity \eqref{eq:abstract-identity}. 

\shuo{
\begin{lemma}\label{lem:alg-identity}
    For any sequence of functions $\{a_n\}_{n\in\mathbb{N}}$, BDF order $k\ge1$ and any $n\ge k$, there exists unique coefficients $\beta_{j\ell}$ for $1\le j\le k$ and $0\le \ell\le k-j$ such that 
    \begin{equation}\label{eq:alg-identity}
    \sum_{j=0}^k \delta_ja_{n-j}^2-2\left(\sum_{j=0}^k\delta_ja_{n-j}\right)\left(\sum_{j=0}^{k-1}\gamma_ja_{n-j-1}\right)=\sum_{j=1}^k\sum_{\ell=0}^{k-j} \beta_{j\ell}s^{2j}(d^j_ta_{n-\ell})^2.
    \end{equation}
\end{lemma}
While requiring only elementary linear algebra and combinatorial identities, the proof contains substantial non-trivial details. 
We defer the proof to Appendix \ref{appendix:appendix-A} to maintain the article's flow, as it involves rather technical arguments. 
Notably, Lemma \ref{lem:alg-identity} remains valid when the quadratic terms of $\{a_n\}_{n\in\mathbb{N}}$ in \eqref{eq:alg-identity} are replaced by arbitrary bilinear operators, as the proof hinges completely on the algebraic structure of bilinearity. 
}

\shuo{
\begin{remark}\label{rmk:value-beta}
    For each BDF-$k$ scheme, we compute the coefficients $\beta_{j\ell}$ by solving the linear system \eqref{eq:linear-system-identity}, which arises in the proof of Lemma \ref{lem:alg-identity}, using the corresponding $\delta_j$ and $\gamma_j$ values from Table \ref{tab:bdf-coefficients}.
    We summarize these values for $k\le4$ as follows. 
    \begin{align}\label{eq:value-beta}
    & k=1:\quad\beta_{10}=1;\quad k=2:\quad\beta_{20}=3/2;\quad k=3:\quad\beta_{30}=11/6,\quad \beta_{21}=-3/2; \\ \nonumber 
    & k=4:\quad\beta_{40}=25/12,\quad\beta_{31}=-10/3,\quad\beta_{21}=-2,\quad\beta_{22}=2.
    \end{align}
    We note that only a few $\beta_{j\ell}$'s are non-zero in each case.
\end{remark}
}

\shuo{
Recalling the definition \eqref{eq:bdf-formula-extrapolation} and \eqref{eq:Backward-formula-expression}, applying Lemma \ref{lem:alg-identity} to $a_n:=\vu^{n}$ and exploiting the fact that $B$ is a symmetric bilinear form, we obtain the following equation for any $n\geq k$ and any $k\ge1$:  
\begin{equation}\label{eq:abstract-identity}
\sum_{j=0}^k \delta_j B(\vu^{n-j},\vu^{n-j})-2sB(\dot{\vu}^n,\hat{\vu}^n)=\sum_{j=1}^k\sum_{\ell=0}^{k-j} \beta_{j\ell}s^{2j}B(d^j_t\vu^{n-\ell},d^j_t\vu^{n-\ell})=:\phi_n.
\end{equation}
Notably, this identity \eqref{eq:abstract-identity} is derived purely algebraically, independent of the iterative methods used for solving \eqref{eq:cons_min}, and is essential to the subsequent analysis.

We let $\{\vu^n\}_{n=k}^{\infty}\subset U$ be a sequence of outputs produced in a BDF-$k$ projection-free iterative scheme. Namely, they satisfy the linearized constraint 
\begin{equation}\label{eq:linear-cons}
B(\dot{\vu}^n,\hat{\vu}^n)=0,
\end{equation}
at each iteration step $n\ge k$. 
We suppose the $k$ initial values $\vu^0,\ldots,\vu^{k-1}$ are given.  
Substituting the coefficients $\delta_j$ from Table~\ref{tab:bdf-coefficients} and $\beta_{j\ell}$ from \eqref{eq:value-beta} into identity \eqref{eq:abstract-identity} yields the algebraic identities as follows. 
When $k=2$, \eqref{eq:abstract-identity} becomes  
\begin{equation}\label{eq:identity-BDF2}
	\frac{3}{2}\B{\vu^n,\vu^n}-2\B{\vu^{n-1},\vu^{n-1}}+\frac{1}{2}\B{\vu^{n-2},\vu^{n-2}}=\frac{3}{2}s^4\B{d_t^2\vu^n,d_t^2\vu^n},
\end{equation}
and for $k=3$ it reads as 
\begin{equation}\label{eq:identity-BDF3}
	\begin{aligned}
		&\frac{11}{6}\B{\vu^n,\vu^n}-3\B{\vu^{n-1},\vu^{n-1}}+\frac32\B{\vu^{n-2},\vu^{n-2}}-\frac13\B{\vu^{n-3},\vu^{n-3}}\\
		&=\frac{11}{6}s^6\B{d_t^3\vu^n,d_t^3\vu^n}-\frac32s^4\B{d_t^2\vu^{n-1},d_t^2\vu^{n-1}},
	\end{aligned}
\end{equation}
and when $k=4$ the identity becomes
{\small
\begin{equation}\label{eq:identity-BDF4}
\begin{aligned}
    &\frac{25}{12}\B{\vu^n,\vu^n}-4\B{\vu^{n-1},\vu^{n-1}}+3\B{\vu^{n-2},\vu^{n-2}}-\frac43\B{\vu^{n-3},\vu^{n-3}}+\frac14\B{\vu^{n-4},\vu^{n-4}}\\
    &=\frac{25}{12}s^8\B{d_t^4\vu^n,d_t^4\vu^n}-\frac{10}{3}s^6\B{d_t^3\vu^{n-1},d_t^3\vu^{n-1}}-2{s^4}\B{d_t^2\vu^{n-1},d_t^2\vu^{n-1}}+2{s^4}\B{d_t^2\vu^{n-2},d_t^2\vu^{n-2}}.
\end{aligned}
\end{equation}
}
}

\subsection{\shuo{Constraint violations}}\label{sec:cons-vio-general}
\shuo{
We denote the constraint violation at each step $n$ as $b^n$, and it is defined by  
\begin{equation}
b_n:=B(\vu^{n},\vu^{n})- \vu_c.
\end{equation}
We aim to establish estimates on $b_n$. 
}

\shuo{
\begin{thm}[Exact expression of constraint violations]\label{thm:bn-expression}
Let $\{\vu^n\}_{n=0}^\infty\subset U$ be a sequence such that the linearized constraint \eqref{eq:linear-cons} holds for any $n\geq k$, where $\dot\vu^n$ and $\hat\vu^n$ are the BDF-$k$ approximation of time derivatives and extrapolation respectively, defined in \eqref{eq:bdf-formula-extrapolation}. Then for every fixed $2\le k\le 6$, we have  
    \begin{equation}\label{eq:diff-eq-sol-2}
    b_N=\sum_{m=1}^{k-1}\left(\tilde\delta_{k-1-m}\sum_{n=k}^N\eta_{N-n}-\sum_{\ell=k-m}^{k-1}\tilde\delta_{\ell}\eta_{N-\ell-m}\right)b_m+\sum_{n=k}^N\eta_{N-n}\sum_{m=k}^n\phi_m,
    \end{equation}
    for $N\geq k$, where $\tilde\delta_{j}$, $\eta_{j}$ and $\phi_m$ are defined in \eqref{eq:def-tilde-delta}, \eqref{eq:taylor-expansion} and \eqref{eq:abstract-identity} respectively.     
\end{thm}
\begin{proof}
We first introduce an auxiliary variable for $n\geq k-1$: 
\begin{equation}\label{eq:def-zn-general}
z_n:=\sum_{j=0}^{k-1} \tilde\delta_j B(\vu^{n-j},\vu^{n-j}) - \vu_c,
\end{equation}
where coefficients are defined in \eqref{eq:def-tilde-delta},  
and it satisfies the following relation for $n\geq k$:
\begin{equation}\label{eq:zn-zn1}
z_n-z_{n-1}=\sum_{j=0}^k \delta_j B(\vu^{n-j},\vu^{n-j}).
\end{equation}
Substituting \eqref{eq:linear-cons} into \eqref{eq:abstract-identity} and using the fact that $\sum_{j=0}^k \delta_j=0$ yield 
\begin{equation}\label{eq:difference-equation}
z_n-z_{n-1}=\sum_{j=0}^k \delta_j b_{n-j}=\phi_n:=\sum_{j=1}^k\sum_{\ell=0}^{k-j} \beta_{j\ell}s^{2j}B(d^j_t\vu^{n-\ell},d^j_t\vu^{n-\ell}),
\end{equation}
for all $n\ge k$. 
Therefore, summing \eqref{eq:difference-equation} over $n$ from $k$ to $N$, we obtain an explicit expression of $z_N$ with $N\geq k$ as 
\begin{equation}\label{eq:express-zN}
z_N=z_{k-1}+\sum_{n=k}^N\phi_n.
\end{equation}
By \eqref{eq:sum-tilde-delta}, we rewrite \eqref{eq:def-zn-general} as  
\begin{equation}\label{eq:zn-diff-eq}
\sum_{j=0}^{k-1} \tilde\delta_j b_{n-j}=z_n,
\end{equation}
with $z_n$ given as \eqref{eq:express-zN}. 
Applying Lemma \ref{lem:diff-eq-sol-general} to the linear inhomogeneous difference equation \eqref{eq:zn-diff-eq}, we obtain an exact expression of $b_N$ as \begin{equation}\label{eq:diff-eq-sol}
    b_N=-\sum_{m=1}^{k-1}\sum_{\ell=k-m}^{k-1}\tilde\delta_{\ell}\eta_{N-\ell-m}b_m+\sum_{n=k}^N\eta_{N-n}z_n,
    \end{equation}
for all $N\ge k$. Then substituting \eqref{eq:express-zN} into \eqref{eq:diff-eq-sol} and using \eqref{eq:zn-diff-eq} for $n=k-1$, we conclude with the expression \eqref{eq:diff-eq-sol-2}.   
\end{proof}
}


\shuo{
We note that the first term in RHS of \eqref{eq:diff-eq-sol-2} is a linear combination of initial values of constraint violations $b_1,\ldots,b_{k-1}$, with coefficients completely determined by BDF coefficients $\delta_0,\ldots,\delta_k$. 
The second term is a linear combination of quadratic terms of discrete time derivatives $d^j_t\vu^{n-\ell}$. 
Theorem \ref{thm:bn-expression} provides an exact expression of constraint violations $b_N$, from which we derive upper bounds of $\|b_N\|_Z$ in the next steps.    
}

\shuo{
\begin{remark}[Initialization and estimates]
    Theorem~\ref{thm:bn-expression} requires \eqref{eq:linear-cons} to hold for $\vu^n$ ($n \geq k$), while the initial values $\vu^1,\ldots,\vu^{k-1}$ influencing the first term in $b_N$ remain unspecified. 
    In practice, we construct initial states $\vu^p$ ($p=1,\ldots,k-1$) to satisfy the BDF-$p$ version of \eqref{eq:linear-cons}:
    \begin{equation}\label{eq:linear-cons-p}
    B(\dot{\vu}^p,\hat{\vu}^p) = 0,
    \end{equation}
    where $\dot{\vu}^p$ and $\hat{\vu}^p$ are the BDF-$p$ approximations of time derivatives and extrapolations respectively from \eqref{eq:bdf-formula-extrapolation} (with $k=p$ in \eqref{eq:bdf-formula-extrapolation}), using the coefficients $\{\delta_j\}_{j=0}^p$ and $\{\gamma_j\}_{j=0}^{p-1}$ defined in \eqref{eq:def-bdf-k-coeff} with $k=p$.
    
    Therefore, in order to obtain the initial values $b_p:=B(\vu^p,\vu^p)-c$ for $p=1,\ldots,k-1$, we apply \eqref{eq:diff-eq-sol-2} to $k=N=p$, and use it recursively for $p=k-1,\ldots,1$. Recalling the definition \eqref{eq:difference-equation} of $\phi_n$, taking norm in $Z$ and applying \eqref{eq:B-upper}, we conclude that 
    \begin{equation}\label{eq:diff-eq-initial}
    \|b_{p}\|_Z\leq\sum_{j=1}^p\sum_{\ell=0}^{p-j} C_{j\ell}s^{2j}\|d^j_t\vu^{p-\ell}\|_U^2, 
    \end{equation}
    where $p=1,\ldots,k-1$ and non-negative constants $C_{j\ell}$ are determined by coefficients of BDF approximations of orders $1,\ldots,p$ (as $\beta_{j\ell}$ appearing in \eqref{eq:difference-equation} and the Taylor expansion coefficients $\eta_j$ both depend on the coefficients of utilized BDF methods).   
\end{remark}
}

\shuo{
We can further exchange the sum in the second term in \eqref{eq:diff-eq-sol-2} as 
\begin{equation}\label{eq:sum-exchange}
\sum_{n=k}^N\eta_{N-n}\sum_{m=k}^n\phi_m=\sum_{m=k}^N\phi_m\left(\sum_{n=m}^N\eta_{N-n}\right). 
\end{equation}
Combining this with the identity \eqref{eq:diff-eq-sol-2}, the estimates \eqref{eq:diff-eq-initial} for initial values and the uniform boundedness \eqref{eq:sum_gamma_bdd}, we derive upper bounds of constraint violations for any BDF-k projection-free methods with $k\le6$ in the following form: 
\begin{equation}\label{eq:cons-upper-general}
\|b_N\|_Z\lesssim \sum_{(j,n)\in\mathcal{J}}s^{2j}\|d^j_t \vu^{n}\|_U^2,
\end{equation}
where $\mathcal{J}\subset\mathbb{N}\times\mathbb{N}$ is an index set. 
We emphasize that no specific iterative schemes have been utilized except for the BDF-k approximated linearized constraint \eqref{eq:linear-cons}, which plays a central role in our analysis. This shows that the proposed methodology is applicable to any projection-free iterative method based on the linearized constraint \eqref{eq:linear-cons}.
}

\shuo{
Building on Theorem~\ref{thm:bn-expression} and the subsequent analysis, we now present explicit constraint violation estimates in the form of \eqref{eq:cons-upper-general} for cases $k=2,3,4$.
\begin{cor}[Case $k=2$]\label{cor:BDF2-cons-vio}
    Suppose $k=2$. Let $\{\vu^n\}_{n=0}^\infty\subset U$ be a sequence such that the BDF-$2$ version of the linearized constraint \eqref{eq:linear-cons} holds for any $n\geq 2$ and the BDF-$1$ version \eqref{eq:linear-cons-p} ($p=1$) holds for $n=1$. 
    Then, the following estimate is valid: 
    \begin{equation}\label{eq:bN-upper with initial velocity}
        \norm{\B{\vu^N,\vu^N}-\vu_c}_Z\lesssim s^2\norm{d_t\vu^0}_U^2+s^4\sum_{n=1}^N\norm{d_t^2\vu^n}_U^2.
    \end{equation}
    If we further assume $d_t\vu^0=\bz$, we have a conditional estimate for constraint violations as follows.
    Suppose for any $N>0$ and $s>0$ the following discrete regularity is satisfied with a constant $c_r>0$: 
    \begin{equation}\label{eq:discrete regularity}
        s\sum_{n=1}^{N}\norm{d_t^2\vu^n}_U^2\leq c_{r}.
    \end{equation}
    Then the constraint violation is of $\order{s^3}$ as 
    \begin{equation}\label{eq:constraint violation O(s^3)}
        \norm{\B{\vu^N,\vu^N}-\vu_c}_Z\lesssim s^3.
    \end{equation}
    \end{cor}
    \begin{proof}
    We first note that the only nonzero element of $\{\beta_{j\ell}\}$ is $\beta_{20}$, according to \eqref{eq:value-beta}. 
    We then derive an estimate for the constraint violation from the exact expression \eqref{eq:diff-eq-sol-2} for the case $k=2$ as follows.  
    By applying the summation exchange technique \eqref{eq:sum-exchange}, utilizing the estimate \eqref{eq:diff-eq-initial} for initial states, and incorporating the uniform boundedness \eqref{eq:sum_gamma_bdd} of Taylor coefficients, we derive
    \begin{equation}\label{eq:cons-upper-BDF2}
        \|b_N\|_Z\lesssim \sum_{n=2}^Ns^{4}\|d^2_t \vu^{n}\|_U^2+s^2\|d_t \vu^{1}\|_U^2. 
    \end{equation}
    Since $d_t\vu^1=d_t\vu^0+sd_t^2\vu^1$, it is clear that $\norm{d_t\vu^1}_U^2\leq 2\norm{d_t\vu^0}_U^2+2s^2\norm{d_t^2\vu^1}_U^2$. Therefore, we can rewrite the estimate \eqref{eq:cons-upper-BDF2} into \eqref{eq:bN-upper with initial velocity}. 
    Moreover, using $d_t\vu^0=\bz$ and the discrete regularity assumption \eqref{eq:discrete regularity} immediately leads to \eqref{eq:constraint violation O(s^3)}.   
    \end{proof}
}

\shuo{Apparently, if $d_t\vu^0\neq\mathbf{0}$, then the term $s^2\norm{d_t\vu^0}_U^2$ in \eqref{eq:bN-upper with initial velocity} limits the upper bound of the constraint violation so that it cannot be better than $\order{s^2}$. This analysis also justifies the choice of zero initial velocity $d_t\vu^0=\mathbf{0}$ throughout the paper.}

\shuo{
\begin{cor}[Case $k=3$]\label{cor:BDF3-cons-vio}
Suppose $k=3$. Let $\{\vu^n\}_{n=0}^\infty\subset U$ be a sequence such that the BDF-$3$ version of linearized constraint \eqref{eq:linear-cons} holds for any $n\geq 3$ and the BDF-$p$ version \eqref{eq:linear-cons-p} with $p=1,2$ holds for $n=1,2$ respectively. Assume $d_t\vu^0=\bz$.  
Then, the following estimate is valid: 
   \begin{equation}\label{eq:bn-upper-BDF3-simpler version}
	\norm{\B{\vu^N,\vu^N}-\vu_c}_Z\lesssim s^4\sum_{n=1}^N\norm{d_t^2\vu^n}_U^2.
   \end{equation}
Moreover, under the discrete regularity assumption \eqref{eq:discrete regularity}, the constraint violation admits the conditional estimate $\order{s^3}$ as in \eqref{eq:constraint violation O(s^3)}.
\end{cor} 
\begin{proof}
For the case $k=3$, the only nonzero elements in $\{\beta_{j\ell}\}$ are $\beta_{30}$ and $\beta_{21}$ as in \eqref{eq:value-beta}. Following the methodology employed in Corollary~\ref{cor:BDF2-cons-vio}, we combine equations \eqref{eq:diff-eq-sol-2}, \eqref{eq:sum-exchange}, \eqref{eq:diff-eq-initial}, and \eqref{eq:sum_gamma_bdd} to derive
\begin{equation}\label{eq:cons-upper-BDF3}
    \|b_N\|_Z\lesssim \sum_{n=3}^N(s^{4}\|d^2_t \vu^{n-1}\|_U^2+s^{6}\|d^3_t \vu^{n}\|_U^2)+s^2\|d_t \vu^{1}\|_U^2+s^4\|d^2_t \vu^{2}\|_U^2. 
\end{equation}
Recalling $d_t\vu^1=sd_t^2\vu^1$ with zero initial velocity and using the estimate $s^{2}\norm{d_t^3\vu^n}_U^2\leq 2\norm{d_t^2\vu^n}_U^2+2\norm{d_t^2\vu^{n-1}}_U^2$, we further simplify the estimate to \eqref{eq:bn-upper-BDF3-simpler version}. The conditional estimate $\order{s^3}$ of constraint violations is readily obtained after employing \eqref{eq:discrete regularity} in \eqref{eq:bn-upper-BDF3-simpler version}.  
\end{proof}
}

\shuo{
\begin{cor}[Case $k=4$]\label{cor:BDF4-cons-vio}
Suppose $k=4$. Let $\{\vu^n\}_{n=0}^\infty\subset U$ be a sequence such that the BDF-$4$ version of linearized constraint \eqref{eq:linear-cons} holds for any $n\geq 4$ and the BDF-$p$ version \eqref{eq:linear-cons-p} with $p=1,2,3$ holds for $n=1,2,3$ respectively. Assume $d_t\vu^0=\bz$.  
Then, the following estimate is valid: 
  \begin{equation}\label{eq:bn-upper-BDF4-simpler version}
	\norm{\B{\vu^N,\vu^N}-\vu_c}_Z\lesssim s^6\sum_{n=3}^N\norm{d_t^3\vu^n}_U^2+s^4\max_{1\leq n\leq N-1}\|d^2_t \vu^{n}\|_U^2.
\end{equation}
Moreover, under the discrete regularity assumption 
\begin{equation}\label{eq:discrete-regularity-BDF4}
\sum_{n=3}^N s^{2} \|d^3_t \vu^{n}\|_U^2 \leq C, \text{ and }\max_{1\leq n\leq N-1}\|d^2_t \vu^{n}\|_U^2\leq C,
\end{equation}
the constraint violation admits the conditional estimate $\order{s^4}$ as 
 \begin{equation}\label{eq:constraint violation O(s^4)}
        \norm{\B{\vu^N,\vu^N}-\vu_c}_Z\lesssim s^4.
    \end{equation}
\end{cor}
\begin{proof}
For $k=4$, the only nonzero elements of $\{\beta_{j\ell}\}$ are $\beta_{40}$, $\beta_{31}$, $\beta_{21}$ and $\beta_{22}$ as in \eqref{eq:value-beta}. In particular, $\beta_{22}=-\beta_{21}=2$ in \eqref{eq:identity-BDF4}, which leads to a telescopic cancellation in parts of $\sum_{m=4}^n\phi_m$, namely 
\begin{align}\label{eq:telescopic-cancellation}
&\sum_{m=4}^n-2s^4B(d^2_t\vu^{m-1},d^2_t\vu^{m-1})+2s^4B(d^2_t\vu^{m-2},d^2_t\vu^{m-2}) \\ \nonumber
&=-2s^4B(d^2_t\vu^{n-1},d^2_t\vu^{n-1})+2s^4B(d^2_t\vu^{2},d^2_t\vu^{2}).
\end{align}
Notably, in order to exploit the cancellation as in \eqref{eq:telescopic-cancellation}, we do not use the exchange of sum \eqref{eq:sum-exchange} for the terms with coefficients $\beta_{21}$ and $\beta_{22}$. We still apply \eqref{eq:sum-exchange} to the other terms with coefficients $\beta_{40}$ and $\beta_{31}$.     
Together with the estimates \eqref{eq:diff-eq-sol-2}, \eqref{eq:diff-eq-initial} and \eqref{eq:sum_gamma_bdd}, this leads to
\begin{align}\label{eq:cons-upper-BDF4}
    \|b_N\|_Z &\lesssim \sum_{n=4}^N(s^{6}\|d^3_t \vu^{n-1}\|_U^2+s^{8}\|d^4_t \vu^{n}\|_U^2)+s^2\|d_t \vu^{1}\|_U^2+s^4\|d^2_t \vu^{2}\|_U^2+s^6\|d^3_t \vu^{3}\|_U^2 \\ \nonumber
    &+\max_{4\leq n\leq N}s^4\|d^2_t \vu^{n-1}\|_U^2.
\end{align}
Noting that $d_t\vu^1=sd_t^2\vu^1$, $s^{2}\norm{d_t^3\vu^3}_U^2\leq 2\norm{d_t^2\vu^3}_U^2+2\norm{d_t^2\vu^{2}}_U^2$, and $s^{2}\norm{d_t^4\vu^n}_U^2\leq 2\norm{d_t^3\vu^n}_U^2+2\norm{d_t^3\vu^{n-1}}_U^2$, 
we further simplify the estimate to \eqref{eq:bn-upper-BDF4-simpler version}. 
Moreover, substituting \eqref{eq:discrete-regularity-BDF4} into \eqref{eq:bn-upper-BDF4-simpler version} yields \eqref{eq:constraint violation O(s^4)}.   
\end{proof}
}

\shuo{
\begin{remark}[Discrete regularity assumptions]\label{rmk:bn-sq}
Formally, \eqref{eq:discrete regularity} represents a discrete analogue of the regularity condition $\partial_{tt}\vu\in L^2(0,T;U)$, where $T:=Ns$ and $\vu^n\sim \vu(ns)$, in the regime as $s\to0$. This is the optimal bound achievable: from \eqref{eq:bN-upper with initial velocity}, $\order{s^4}$ convergence would demand uniform boundedness of $\sum_{n=1}^N\norm{d_t^2\vu^n}_U^2$, which typically diverges as $s\to0$.

Notably, the constraint violation estimates for BDF-2 and BDF-3 differ only by a constant, suggesting no inherent advantage of BDF-3 over BDF-2 in constraint consistency. Moreover, \eqref{eq:discrete-regularity-BDF4} approximates the conditions $s^{1/2}\partial_{ttt}\vu\in L^2(0,T;U)$ and $\partial_{tt}\vu\in L^{\infty}(0,T;U)$ asymptotically.

Section~\ref{sec:numerical results} will demonstrate that our BDF-$k$ accelerated gradient flow methods (that will be discussed in the subsequent sections) achieve the predicted $\order{s^3}$ (BDF-2,3) and $\order{s^4}$ (BDF-4) constraint violations, confirming that \eqref{eq:discrete regularity} and \eqref{eq:discrete-regularity-BDF4} hold for our test cases.
\end{remark}
}

\section{\shuo{Accelerated flows with BDF schemes and linearized constraints}}\label{sec:BDF2}
\setcounter{equation}{0}
\subsection{\shuo{BDF-1 version and its properties}}\label{sec:BDF1}
\shuo{We first introduce the BDF-1 projection-free accelerated gradient flow method to solve the constrained minimization problem \eqref{eq:cons_min}.
Recall that for the BDF-1 approximation, $\dot\vu^n=d_t\vu^n=s^{-1}(\vu^n-\vu^{n-1})$ and $\hat\vu^n=\vu^{n-1}$. 
Motivated by the second order dynamics \eqref{eq:second order flow problem}, we propose the iterative scheme as follows. Given $\vu^0\in\mathcal{A}$, for $n\ge1$ we compute $d_t\vu^n\in\mathcal{F}_{\vu^{n-1}}$ such that 
 \begin{equation}\label{eq:discrete_flow}
			\innerprocuct{d_t^2\vu^n,\vw}_U+\frac{\alpha}{t_n}\innerprocuct{d_t\vu^n,\vw}_U+\mathcal{M} (\vu^{n-1}+sd_t\vu^n,\vw)=0,
\end{equation}
for all $\vw\in\mathcal{F}_{\vu^{n-1}}$, and update $\vu^n=\vu^{n-1}+sd_t\vu^n$. We take $t_n:=ns$ and the initial velocity $d_t\vu^0=\bz$. 
Notably, the first step of \eqref{eq:discrete_flow} with $n=1$ is a step of gradient flow with an additional constant $(1+\alpha)/s$ multiplied to the flow metric term $\innerprocuct{d_t\vu^n,\vw}_U$.    
We summarize the proposed method in Algorithm \ref{Algorithm 1}.  
}
  
    \begin{algorithm}
		\leavevmode\newline(0)\textbf{Input}: $\vu^0\in\mathcal{A}$, $d_t\vu^0=\bz$, $s>0$, $\alpha\ge3$, $\varepsilon>0$ and a $T_{max}>0$ that is sufficiently large. \textbf{Set} $n=1$.\\
		(1)\textbf{Solve} $d_t\vu^n\in\mathcal{F}_{\vu^{n-1}}$ by \eqref{eq:discrete_flow}. \\
		(2)\textbf{Update} $\vu^n=\vu^{n-1}+sd_t\vu^n$. \\
		(3)\textbf{Stop when} $\seminorm{\frac12d_t\left.\left( \M{\vu^n,\vu^n}+\norm{d_t\vu^n}^2_U\right)\right.}\leq \varepsilon$ or $n>[T_{max}/s]$, \textbf{ otherwise set} $n\to n+1$ and \textbf{continue} with step (1). \\
        (4)\textbf{Output}: $\vu^n$.
        \caption{BDF1 discretization of accelerated gradient flow with tangent space update}\label{Algorithm 1}
	\end{algorithm}

\shuo{
We now establish the total energy stability and unconditional constraint violation for the proposed method. We view $E[\vu^n]=\frac12\M{\vu^n,\vu^n}$ as the potential energy, $\frac12\norm{d_t\vu^n}^2_U$ as the kinetic energy and the sum of them as the total energy. 
\begin{thm}\label{thm:BDF1-property}
 Let $\{\vu^n\}_{n\in\mathbb{N}}$ be a sequence of outputs produced in Algorithm \ref{Algorithm 1}. \\
		(a) \textup{Energy decay}: For every $n\geq 1$, 
  \begin{equation}\label{eq:energy-decay-BDF1}
		\M{\vu^n,\vu^n}+\norm{d_t\vu^n}^2_U\leq \M{\vu^{n-1},\vu^{n-1}}+\norm{d_t\vu^{n-1}}^2_U,
  \end{equation}
  and for $N\geq 1$
  	\begin{equation}\label{eq:sum-BDF1}
		\begin{aligned}
	\M{\vu^N,\vu^N}+\norm{d_t\vu^N}^2_U+s^2\sum_{n=1}^{N}\left(\norm{d^2_t\vu^n}^2_U+\M{d_t\vu^n,d_t\vu^n}+\frac{\alpha}{ns^2}\norm{d_t\vu^n}^2_U\right)\leq\M{\vu^0,\vu^0}.
		\end{aligned}		
	\end{equation}
		(b) \textup{Unconditional constraint violation}: 
\begin{equation}\label{eq:BDF1-cons-vio}
     \norm{B(\vu^N,\vu^N)-\vu_c }_Z\lesssim Ns^2.
 \end{equation}
\end{thm}
The proof follows by a straightforward generalization of the theory in \cite{DonGuoYan24}, and thus we omit details of the proof and only outline the key steps as follows. Taking $\vw = d_t\vu^n$ in \eqref{eq:discrete_flow} yields the energy stability. The summability of $n^{-1}\|d_t\vu^n\|_U^2$ in \eqref{eq:sum-BDF1} implies \eqref{eq:BDF1-cons-vio}. 

For a fixed stopping tolerance $\varepsilon$ in Algorithm~\ref{Algorithm 1}, the total iteration count satisfies $N \approx \order{s^{-1}}$ (see \cite{DonGuoYan24} for detailed discussion and numerical verification). Consequently, the constraint violation \eqref{eq:BDF1-cons-vio} simplifies to
\begin{equation}\label{eq:BDF1-cons-vio-2}
    \norm{B(\vu^N,\vu^N) - \vu_c}_Z \lesssim s.
\end{equation}
}

\subsection{\shuo{BDF-2 version and its adapted stability}}\label{sec:bdf2-stability}

\shuo{We next present the BDF-2 projection-free accelerated gradient flow scheme to solve the constrained minimization problem \eqref{eq:cons_min}, and it employs BDF-2 approximations to enforce the linearized constraint \eqref{eq:linear-cons}.}

\shuo{Recalling the definition \eqref{eq:bdf-formula-extrapolation} and \eqref{eq:def-bdf-k-coeff}, the BDF-2 approximation of $\dot{\vu}(t)$ and corresponding extrapolation (approximation of $\vu(t)$) at $t=t_n:=ns$ are
\begin{equation}\label{eq:BDF-2}
	\dot{\vu}^n=\frac{3\vu^n-4\vu^{n-1}+\vu^{n-2}}{2s},\quad \hat{\vu}^n:=2\vu^{n-1}-\vu^{n-2}.
\end{equation}
}

\shuo{We propose the BDF-2 accelerated gradient flows as follows.}
We compute $\dot{\vu}^n$ in the tangent space $ \mathcal{F}_{\hat{\vu}^n}$ through the following equation 
\begin{equation}\label{eq:discrete_flow_BDF2}
	\innerprocuct{(\dot{\vu}^n-\dot\vu^{n-1})/s,\vw}_U+\innerprocuct{\frac{\alpha}{t_n}\dot{\vu}^n,\vw}_U+\frac13\mathcal{M} \left(4\vu^{n-1}-\vu^{n-2}+2s\dot{\vu}^n,\vw\right)=0,
\end{equation}
for all $\vw\in\mathcal{F}_{\hat{\vu}^n}$, \shuo{which is motivated by \eqref{eq:second order flow problem}.} 
It is worthy noting that the first term in \eqref{eq:discrete_flow_BDF2} requires $n\ge3$. Therefore, we initialize the flow by solving $\vu^1$ with one step of the BDF-1 scheme \shuo{as follows.  
We compute $d_t\vu^1\in\mathcal{F}_{\vu^{0}}$ such that 
 \begin{equation}\label{eq:BDF2-initial}
			\innerprocuct{d_t^2\vu^1,\vw}_U+\frac{\alpha}{s}\innerprocuct{d_t\vu^1,\vw}_U+\mathcal{M} (\vu^{0}+sd_t\vu^1,\vw)=0,
\end{equation}
for all $\vw\in\mathcal{F}_{\vu^{0}}$ and with the convention $d_t\vu^0=\bz$. We then update $\vu^1=\vu^{0}+sd_t\vu^1$.  
Moreover, we} compute $\dot{\vu}^2 \in \mathcal{F}_{\hat{\vu}^2}$ such that 
\begin{equation}\label{eq:transition_step}
	\innerprocuct{(\dot{\vu}^2-d_t\vu^1)/s,\vw}_U+\frac{\alpha}{2s}\innerprocuct{\dot{\vu}^2,\vw}_U+\frac13\mathcal{M} \left(4\vu^{1}-\vu^{0}+2s\dot{\vu}^2,\vw\right)=0,\; \text{ for all } \;\vw\in\mathcal{F}_{\hat{\vu}^2}.
\end{equation}
We summarize the proposed method in Algorithm \ref{Algorithm 2}.

\begin{algorithm}
		\leavevmode\newline(0)\textbf{Input}: $\vu^0\in\mathcal{A}$ (in particular $\B{\vu^0,\vu^0}=\vu_c$), $s>0$, $d_t\vu^0=\mathbf{0}$, $\alpha\ge3$, $\varepsilon>0$ and a $T_{max}>0$ that is sufficiently large.\\
        (1)\textbf{Initialization}: Compute $d_t\vu^1\in\mathcal{F}_{\vu^{0}}$ by \eqref{eq:BDF2-initial} with $n=1$ . \textbf{Update} $\vu^1=\vu^{0}+sd_t\vu^1$.\\
        (2)\textbf{Transition step}: Set $\hat{\vu}^2=2\vu^1-\vu^0$ and compute $\dot{\vu}^2 \in \mathcal{F}_{\hat{\vu}^2}$ by \eqref{eq:transition_step}. \textbf{Update} $\vu^2=(4\vu^{1}-\vu^{0}+2s \dot{\vu}^2)/3$. \textbf{Set} $n=3$. \\    
		(3)\textbf{Compute} $\dot{\vu}^n\in\mathcal{F}_{\hat{\vu}^n}$ by \eqref{eq:discrete_flow_BDF2} with $\hat{\vu}^n=2\vu^{n-1}-\vu^{n-2}$. \\
		(4)\textbf{Update} $\vu^n=\left(4\vu^{n-1}-\vu^{n-2}+2s\dot{\vu}^n\right)/3$. \\
		(5)\textbf{Stop when} $\seminorm{\frac{d_t}{2}\left.\left( \GM{\vu^n,\vu^{n-1}}+\norm{\dot{\vu}^n}^2_U\right)\right.}\leq \varepsilon$ or $n>[T_{max}/s]$, \textbf{ otherwise set} $n\to n+1$ and \textbf{continue} with step (3). \\
        (6)\textbf{Output}: $\vu^n$.
        \caption{BDF-2 discretization of accelerated gradient flows with tangent space update}\label{Algorithm 2}
	\end{algorithm}
 
 The BDF-2 scheme exhibits an energy stability property known as $G$-stability in the context of numerical solutions for ODEs, as discussed in \cite{wanner1996solving}. The following identity is crucial in proving the energy stability of the BDF-2 scheme for the accelerated gradient flow in this section
\begin{equation}\label{eq:A-GA}
	2s\mathcal{M} (\dot{\vu}^n,\vu^n)=\GM{\vu^n,\vu^{n-1}}-\GM{\vu^{n-1},\vu^{n-2}}+\frac{s^4}{2}\mathcal{M}(d_t^2\vu^n,d_t^2\vu^n).
\end{equation}
Here $\mathcal{G}_{\mathcal{M}}$ is another bilinear form depending on $\mathcal{M}$ as
\begin{equation}\label{eq:GA}
	\mathcal{G}_{\mathcal{M}}(\vu,\vv)=g_{11}\mathcal{M}(\vu,\vu)+g_{12}\mathcal{M}(\vu,\vv)+g_{21}\mathcal{M}(\vv,\vu)+g_{22}\mathcal{M}(\vv,\vv),
\end{equation}
for arbitrary $\vu,\vv\in U$ with the coefficients $g_{11}=5/2$, $g_{12}=g_{21}=-1$, $g_{22}=1/2$.
Then substituting the identity $2\mathcal{M} (\vu,\vv)=\M{\vu,\vu}+\M{\vv,\vv}-\M{\vu-\vv,\vu-\vv}$ into \eqref{eq:GA} leads to an equivalent expression of $\GM{\vu,\vv}$:
\begin{equation}\label{eq:GA_2}
	\mathcal{G}_{\mathcal{M}}(\vu,\vv):=\mathcal{M}(\vu-\vv,\vu-\vv)+\frac{3}{2}\mathcal{M}(\vu,\vu)-\frac{1}{2}\mathcal{M}({\vv,\vv}).
\end{equation}
One can verify \eqref{eq:A-GA} via a direct calculation using \eqref{eq:GA_2}, as well as the definitions \eqref{eq:Backward-formula-higher} and \eqref{eq:BDF-2} of $d_t^2\vu^n$ and $\dot{\vu}^n$.
Moreover, the symmetric matrix $G=(g_{ij})_{i,j=1}^2$ is positive definite and in particular has two positive eigenvalues $\lambda_1=(3-2\sqrt{2})/2$ and $\lambda_2=(3+2\sqrt{2})/2$. Therefore, $\GM{\vu,\vv}$ can be further estimated as follows: 
\begin{equation}\label{eq:GA-equivalence}
	0\leq\lambda_1\left(\M{\vu,\vu}+\M{\vv,\vv}\right)\leq  \mathcal{G}_{\mathcal{M}}(\vu,\vv)\leq \lambda_2\left(\M{\vu,\vu}+\M{\vv,\vv}\right).
\end{equation}

We proceed now to prove energy stability of the BDF-2 scheme for the accelerated gradient flow. 
The following propositions give useful estimates for the initialization step and transition step. 
\begin{proposition}[Initialization step]\label{prop:init-step}
Given $\vu^0\in\mathcal{A}$ and $d_t\vu^0=\mathbf{0}$, let $d_t\vu^1\in\mathcal{F}_{\vu^{0}}$ be the solution to \eqref{eq:BDF2-initial} with $n=1$ and $\vu^1=\vu^{0}+sd_t\vu^1$. Recall that $c_Z$ is the constant appearing in \eqref{eq:B-upper}. 
Then the following estimates are valid:
\begin{align}
	&\norm{d_t\vu^1}^2_U\leq \frac{1}{2\alpha}\M{\vu^0,\vu^0} , \label{eq:init-est-1} \\
	&\mathcal{G}_{\mathcal{M}}(\vu^1,\vu^0)\leq 2\lambda_2\M{\vu^0,\vu^0}. \label{eq:init-est-2} 
 \end{align}
\end{proposition}
\begin{proof}
\shuo{Note that} the initialization step in Algorithm \ref{Algorithm 2} is a step of BDF-1 scheme. 
\shuo{Taking $\vw=d_t\vu^1$ in \eqref{eq:BDF2-initial}} leads to \eqref{eq:init-est-1} and also the fact that $\M{\vu^1,\vu^1}\leq \M{\vu^0,\vu^0}$. 
Then combining this with \eqref{eq:GA-equivalence}, we derive
\begin{equation}\label{eq:GM-est}
\GM{\vu^1,\vu^0}\leq \lambda_2(\M{\vu^0,\vu^0}+\M{\vu^1,\vu^1})\leq 2\lambda_2\M{\vu^0,\vu^0},
\end{equation}
which validates \eqref{eq:init-est-2}. 
\end{proof}

\begin{proposition}[Transition step]\label{prop:trans-step}
Given $\vu^0\in\mathcal{A}$ and $d_t\vu^0=\bz$, let $\vu^1$ be the output of the initialization step in Algorithm \ref{Algorithm 2}, $\dot{\vu}^2 \in \mathcal{F}_{\hat{\vu}^2}$ be the solution of \eqref{eq:transition_step} with $\hat{\vu}^2=2\vu^1-\vu^0$ and $\vu^2:=(4\vu^{1}-\vu^{0}+2s \dot{\vu}^2)/3$. Then the following estimates are valid:
\begin{align}
	&\norm{\dot{\vu}^2}^2_U+\mathcal{G}_{\mathcal{M}}(\vu^2,\vu^1)\leq \left(2\lambda_2+\frac{1}{2\alpha}\right)\M{\vu^0,\vu^0}, \label{eq:transition-est-1} \\
	&\norm{\dot{\vu}^2}^2_U\leq \left(\frac{1}{2\alpha^2}+\frac{2\lambda_2}{\alpha}\right)\M{\vu^0,\vu^0}. \label{eq:transition-est-2}
 \end{align}
\end{proposition}
\begin{proof}
Taking $\vw=\dot{\vu}^2$ in \eqref{eq:transition_step}, and using \eqref{eq:A-GA} and $\vu^2=(4\vu^{1}-\vu^{0}+2s \dot{\vu}^2)/3$, we derive 
\begin{equation}\label{eq:transition-est-inter}
\left(\GM{\vu^2,\vu^1}+\norm{\dot{\vu}^2}^2_U\right)+\norm{\dot{\vu}^2-d_t\vu^1}^2_U+\alpha\norm{\dot{\vu}^2}^2_U+\frac{s^4}{2}\M{d^2_t\vu^2,d^2_t\vu^2}=\GM{\vu^1,\vu^0}+\norm{d_t\vu^1}^2_U.
\end{equation}
Together with Proposition \ref{prop:init-step}, this implies \eqref{eq:transition-est-1} and \eqref{eq:transition-est-2}.
\end{proof}

\shuo{The following theorem establishes that Algorithm \ref{Algorithm 2} generates iterations which monotonically decrease a modified ``total energy'' functional}, namely 
\[\GM{\vu^n,\vu^{n-1}}+\norm{\dot{\vu}^n}_U^2,\]
\shuo{where $\norm{\dot{\vu}^n}_U^2$ represents a discrete kinetic energy and $\GM{\vu^n,\vu^{n-1}}$ can be viewed as a modified potential energy.}
Moreover, it is proved that the iterations terminate in finitely many steps.
\begin{thm} 
Let $\{\vu^n\}_{n\in\mathbb{N}}$ be a sequence of the outputs produced in Algorithm \ref{Algorithm 2}. \\
	(a) \textup{(Energy decay).} For every $n\geq 3$, we have
    \begin{equation}\label{eq:GM-stab}
	\GM{\vu^n,\vu^{n-1}}+\norm{\dot{\vu}^n}^2_U\leq \GM{\vu^{n-1},\vu^{n-2}}+\norm{\dot{\vu}^{n-1}}^2_U.
    \end{equation}
    (b) \textup{(Termination in finite steps).} For every given $s>0$ and $\varepsilon>0$, Algorithm \ref{Algorithm 2} terminates in a finite number of iterations.	
\end{thm}
\begin{proof} 
Choosing $\vw=\dot{\vu}^n$ in \eqref{eq:discrete_flow_BDF2}, using the identity \eqref{eq:A-GA} and recalling the definition \eqref{eq:BDF-2}, we obtain 
\begin{equation}\label{eq:BDF2-stab-onestep}
\begin{aligned}
	\left(\GM{\vu^n,\vu^{n-1}}+\norm{\dot{\vu}^n}^2_U\right)+\frac{2\alpha}{n}\norm{\dot{\vu}^n}^2_U&\\
	+s^2\norm{d_t\dot{\vu}^n}^2_U+\frac{s^4}{2}\M{d_t^2\vu^n,d_t^2\vu^n}& =\GM{\vu^{n-1},\vu^{n-2}}+\norm{\dot{\vu}^{n-1}}^2_U,
\end{aligned}
\end{equation}
which immediately implies \eqref{eq:GM-stab}. Then summing \eqref{eq:BDF2-stab-onestep} from $n=3$ to $N$ on both sides and using Proposition \ref{prop:trans-step} yield 
\begin{equation}\label{eq:BDF2-stab-sum}
\begin{aligned}
		\left(\GM{\vu^N,\vu^{N-1}}+\norm{\dot{\vu}^N}^2_U\right) +\sum_{n=3}^{N}\frac{2\alpha}{n}\norm{\dot{\vu}^n}^2_U&\\
		+ s^2\sum_{n=3}^{N}\norm{d_t\dot{\vu}^n}^2_U +\frac{s^4}{2}\sum_{n=3}^{N}\M{d_t^2\vu^n,d_t^2\vu^n}& \leq\left(2\lambda_2+\frac{1}{2\alpha}\right)\M{\vu^0,\vu^0}.
	\end{aligned}
 \end{equation}
\shuo{
This further implies that $(1/n)\norm{\dot\vu^n}_U^2$, $\norm{d_t\dot\vu^n}_U^2$ and $\norm{d^2_t\vu^n}_U^2$ converges to $0$ as $n\to\infty$. 
By \eqref{eq:BDF2-stab-onestep}, the stopping criteria in Algorithm \ref{Algorithm 2} requires that
	\begin{equation}\label{equation for termination}
\varepsilon\geq\seminorm{\frac12d_t\left(\GM{\vu^n,\vu^{n-1}}+\norm{d_t\vu^n}^2_U\right)}=\frac{\alpha}{ns}\norm{\dot{\vu}^n}^2_U+\frac{s}{2}\norm{d_t\dot{\vu}^n}^2_U+\frac{s^3}{4}\M{d_t^2\vu^n,d_t^2\vu^n},
	\end{equation}
 for a given $\varepsilon>0$. Since the RHS converges to $0$ as $n\to\infty$, there exists a $N\in\mathbb{N}$ depending only on $\varepsilon$ such that the stopping criteria is guaranteed when $n\ge N$, which finishes the proof of (b).
 }      
\end{proof}

\shuo{
Clearly, as the stopping tolerance $\varepsilon\to0$ in Algorithm \ref{Algorithm 2}, the total iteration number tends to infinity. 
As in the following corollary, the energy stability guarantees both convergence of the scheme and a vanishing kinetic energy in the limit $n\to\infty$. 
\begin{cor}\label{cor:limit-BDF2}
Let $\{\vu^n\}_{n\in\mathbb{N}}$ be a sequence of the outputs produced in Algorithm \ref{Algorithm 2} and $s>0$ be fixed. \\
(a) \textup{(Existence of limit).} There exists a subsequence of $\vu^n$ (without relabeling) and $\vu^*\in U$ such that $\vu^n$ converges weakly to $\vu^*$ in $U$. \\
(b) \textup{(Vanishing kinetic energy in the limit).} As $n\to\infty$, $\dot\vu^n$ converges strongly to $\bz$ in $U$, up to a subsequence.  
\end{cor}
\begin{proof}
By the energy stability \eqref{eq:BDF2-stab-sum}, the equivalence estimate \eqref{eq:GA-equivalence} for $\mathcal{G}_{\mathcal{M}}$ and the coercivity \eqref{assum:bilinear} of $\mathcal{M}$, we conclude that $\|\vu^n\|_U$ is uniformly bounded. Whence, there exists a subsequence of $\vu^n$ and $\vu^*\in U$ such that $\vu^n$ converges weakly to $\vu^*$ in the Hilbert space $U$. 
Moreover, by \eqref{eq:BDF2-stab-sum} we note that $\sum_{n=3}^{\infty}\frac{\norm{\dot{\vu}^n}^2_U}{n} \leq C$, which immediately leads to (b).  
\end{proof}
}

\shuo{We now examine the limit point $\vu^*$ obtained in Corollary \ref{cor:limit-BDF2}.
In practice, to realize the linearized constraint \eqref{eq:linear-cons}, we introduce Lagrange multipliers $\lambda^n\in Z^{*}$, rewrite \eqref{eq:discrete_flow_BDF2} into the following saddle-point system, and compute both the multiplier and $\dot\vu^n\in U_0$:
\begin{align}\label{eq:discrete_flow_saddle}
			\innerprocuct{d_t\dot\vu^n,\vw}_U+\innerprocuct{\frac{\alpha}{t_n}\dot{\vu}^n,\vw}_U+\frac13\mathcal{M} \left(4\vu^{n-1}-\vu^{n-2}+2s\dot{\vu}^n,\vw\right)+\langle\lambda^n,B(\hat\vu^n,\vw)\rangle&=0, \\ \nonumber
            \langle\mu,B(\hat\vu^n,d_t\vu^n)\rangle&=0,
\end{align}
for any test function $(\vw,\mu)\in U_0\times Z^{*}$. Here, $\langle\cdot,\cdot\rangle$ denotes the dual action between $Z^{*}$ and $Z$. The classical theory on saddle-point systems \cite{boffi2013mixed} guarantees the solvability of \eqref{eq:discrete_flow_saddle} when the following inf-sup stability condition holds: Assume $B$, $Z$ and $U$ are given such that for any $n\ge0$ there exists a constant $c>0$  
\begin{equation}\label{eq:inf-sup}
\inf_{\mu\in Z^{*}}\sup_{\vw\in U_0}\frac{\langle\mu,B(\hat\vu^n,\vw)\rangle}{\|\vw\|_U\|\mu\|_{Z^{*}}} \geq c. 
\end{equation}
We characterize $\vu^*$ in the next Corollary. 
 \begin{cor}\label{cor:convergence}
 Let $\{\vu^n\}_{n=0}^N$ be a sequence of outputs produced in Algorithm \ref{Algorithm 2} with \eqref{eq:discrete_flow_BDF2} replaced by \eqref{eq:discrete_flow_saddle}. For a fixed $s>0$, $N:=N(\varepsilon)$ denotes the final iteration number determined by the stopping tolerance $\varepsilon$. When $N\to\infty$ (i.e., $\varepsilon\to0$), there exists $\vu^{*}\in U$ such that $\vu^n\rightharpoonup\vu^{*}$ weakly in $U$ up to a subsequence. Moreover, if we further assume the inf-sup stability \eqref{eq:inf-sup}, $Z\subset Z^{**}$, and that $B(\cdot,\vw):U\to Z$ is continuous with respect to the weak topology of $U$ for any $\vw\in U_0$, then $\vu^{*}$ is a local minimizer of $E$ in the tangent space $\mathcal{F}_{\vu^{*}}$, i.e., 
 \begin{equation}\label{eq:local-min-in-tangent}
 E[\vu^{*}]\le E[\vu^{*}+\vw]\quad\forall \vw\in\mathcal{F}_{\vu^{*}}. 
 \end{equation}
\end{cor}
\begin{proof}
By Corollary \ref{cor:limit-BDF2} there exists $\vu^{*}\in U$ such that $\vu^n\rightharpoonup\vu^{*}$ weakly in $U$ up to a subsequence (without relabeling).
From the summability of $n^{-1}\|\dot\vu^n\|_U^2$, $\|d_t\dot\vu^n\|_U^2$ and $\mathcal{M}(d^2_t\vu^n,d^2_t\vu^n)$ given in \eqref{eq:BDF2-stab-sum}, as well as the coercivity of $\mathcal{M}$, we conclude that $n^{-1/2}\|\dot\vu^n\|_U\to0$, $\|d^2_t\vu^n\|_U\to0$ and $\|d_t\dot\vu^n\|_U\to0$ as $n\to\infty$. Moreover, as $\mathcal{M}(\cdot,\vw)\in U^{*}$ and $\vu^n=4\vu^{n-1}-\vu^{n-2}+2s\dot{\vu}^n$, one can pass to the limits in the first equation of \eqref{eq:discrete_flow_saddle} to obtain 
\begin{equation}\label{eq:pass-limit-1}
\lim_{n\to\infty}\langle\lambda^n,B(\hat\vu^n,\vw)\rangle=-\mathcal{M}(\vu^{*},\vw)\quad\forall\vw\in U_0.
\end{equation}
This in conjunction with the inf-sup condition \eqref{eq:inf-sup} yields
\begin{equation}\label{eq:boundedness-lambda}
\sup_{n\ge0}\|\lambda^n\|_{Z^{*}}\le\frac{1}{\beta}\sup_{n\ge0}\sup_{\vw\in U_0}\frac{\langle\lambda^n,B(\hat\vu^n,\vw)\rangle}{\|\vw\|_U}<\infty.
\end{equation}
Applying the Banach-Alaoglu theorem for the space $Z^{*}$, we conclude that there exists $\lambda^{*}\in Z^{*}$ such that $\lambda^n\to\lambda^{*}$ in $Z^{**}$ topology up to a subsequence (without relabeling). In order to pass limits in $\hat\vu^n$, we note that 
\begin{equation}\label{eq:uhatn-rewrite}
\hat\vu^n=\vu^n-\vu^n+\hat\vu^n=\vu^n-s^2d^2_t\vu^n,
\end{equation}
and $\hat\vu^n\rightharpoonup\vu^*$ weakly up to a subsequence, due to the weak convergence of $\vu^n\rightharpoonup\vu^*$ (up to a subsequence) and the strong convergence of $d^2_t\vu^n$ to $\bz$.  
Moreover, since $B(\cdot,\vw)$ is continuous with respect to the weak topology of $U$, $B(\hat\vu^n,\vw)\to B(\vu^*,\vw)$ strongly in $Z$. 
Together with the fact that $B(\vu^*,\vw)\in Z\subset Z^{**}$, we conclude that $\lim_{n\to\infty}\langle\lambda^n,B(\hat\vu^n,\vw)\rangle=\langle\lambda^*,B(\vu^*,\vw)\rangle$, which is $0$ if $\vw\in\mathcal{F}_{\vu^{*}}$, according to the definition \eqref{eq:tangent_F}.
Substituting it into \eqref{eq:pass-limit-1} leads to 
\begin{equation}\label{eq:limit-ustar}
\delta E[\vu^{*}](\vw)=\mathcal{M}(\vu^{*},\vw)=0,
\end{equation}
for all $\vw\in\mathcal{F}_{\vu^{*}}$. As $E$ is quadratic, this verifies \eqref{eq:local-min-in-tangent} and concludes the proof.
\end{proof}

\begin{remark}\label{rmk:convergence-discrete}
Corollary \ref{cor:convergence} is a generalized and spatial-continuous version of \cite[Proposition 5.5]{bonito2023numerical}. In a fully-discretized version of Corollary \ref{cor:convergence}, strong convergence can be achieved, and the assumptions that $Z \subset Z^{**}$ and the continuity of $B$ are no longer necessary, since all norms are equivalent in finite-dimensional spaces. However, a discrete version of the inf-sup stability condition \eqref{eq:inf-sup} remains necessary.
\end{remark}
}

\subsection{\shuo{Constraint violations for the BDF-2 scheme}}\label{sec:cons-vio-BDF2}
\shuo{We first} prove an estimate for discrete time derivatives that \shuo{leads to the discrete regularity needed in the proof of constraint violation.}  
Via definitions of finite difference quotients, we show that the norms of $d_t\vu^n$ and $d^2_t\vu^n$ can be bounded by quantities in terms of $\dot{\vu}^n$ and $d_t\dot{\vu}^n$ respectively in the following lemma. Note that (a) of Lemma \ref{lem:norm-dtu-d2tu} is already proved in \cite{bartelsakrivis2023quadratic}, while \eqref{eq:norm-d2tu-1} is new. 
\begin{lemma}\label{lem:norm-dtu-d2tu}
	(a) For every integer $n\geq 1$,
	\begin{equation}\label{eq:norm-dtu-1}
		\norm{d_t\vu^n}^2_U\leq\frac{8}{9}\norm{\dot{\vu}^n}^2_U+\frac{2}{9}\norm{d_t\vu^{n-1}}^2_U.
	\end{equation}
	(b) For every integer $N\geq 3$,
	\begin{equation}\label{eq:norm-d2tu-1}
		\sum_{n=3}^{N}\norm{d_t^2\vu^n}^2_U\leq \frac{8}{7}\sum_{n=3}^{N}\norm{d_t\dot{\vu}^n}^2_U+\frac{2}{7}\norm{d_t^2\vu^2}^2_U.
	\end{equation}
\end{lemma}
 \begin{proof}
The relation $2\dot{\vu}^n=3d_t\vu^n-d_t\vu^{n-1}$ \shuo{and Cauchy-Schwarz inequality implies \eqref{eq:norm-dtu-1}.}
Taking $d_t$ on both sides of $2\dot{\vu}^n=3d_t\vu^n-d_t\vu^{n-1}$ leads to $2d_t\dot \vu^n=3d_t^2\vu^n-d_t^2\vu^{n-1}$. \shuo{Then we apply Lemma \ref{lem:diff-eq-sol-sharper-estimate} with $a_n=d_t^2\vu^n$ and $f_n=d_t\dot \vu^n$ to obtain \eqref{eq:norm-d2tu-1}.} 
\end{proof}

\begin{remark}\label{remark:discrete-regularity}
Compared with the BDF-2 scheme for gradient flows in \cite{bartelsakrivis2023quadratic}, the BDF-2 scheme for accelerated gradient flow inherits stronger regularity in terms of temporal derivatives. More specifically, the method in \cite{bartelsakrivis2023quadratic} only guarantees an $\mathcal{O}(s^{-1})$ upper bound for $\norm{d_t\vu^n}_U^2$ or $\norm{\dot{\vu}^n}_U^2$, while the BDF-2 scheme for accelerated gradient flow  ensures an $\mathcal{O}(1)$ bound for them. 
More precisely, the sequences $\left\{\norm{\dot{\vu}^n}_U\right\}_{n\geq 2}$, $\left\{\norm{d_t\vu^n}_U\right\}_{n\geq 1}$ and $\left\{s\norm{d_t^2\vu^n}_U\right\}_{n\geq 2}$ all are uniformly bounded. To see this, we first notice that \eqref{eq:BDF2-stab-sum} guarantees the uniform boundedness of $\norm{\dot{\vu}^n}_U$. By Proposition \ref{prop:init-step}, $\norm{d_t\vu^1}_U$ is bounded. Using \eqref{eq:norm-dtu-1} recursively we conclude that $\norm{d_t\vu^n}_U$ is also uniformly bounded. The relation $\dot{\vu}^n-d_t\vu^n=(s/2)d_t^2\vu^n$ and the uniform boundedness of $\norm{d_t\vu^n}_U$ and $\norm{\dot{\vu}^n}_U$ immediately ensure the uniform boundedness of $s\norm{d_t^2\vu^n}_U$.
\end{remark}

\shuo{
The following theorem states an unconditional estimate of constraint violations for the BDF-2 scheme. We refer to Corollary \ref{cor:BDF2-cons-vio} for a conditional estimate. 
\begin{thm}\label{thm:BDF2_cons_vio_1}
    Let $\{\vu^n\}_{n\in\mathbb{N}}$ be a sequence of outputs produced by Algorithm \ref{Algorithm 2}. For every integer $N\geq 1$, it satisfies that 
 \begin{equation}\label{eq:bn-BDF2-orderO(s^2) estimate}
     \norm{B(\vu^N,\vu^N)-\vu_c }_Z\lesssim s^2.
 \end{equation}
\end{thm}
\begin{proof}
Using Lemma \ref{lem:norm-dtu-d2tu} (a) for $n=2$ and the Cauchy-Schwarz inequality, we estimate
\begin{equation}\label{eq:d2tu2}
		s^2\norm{d_t^2\vu^2}^2_U=\norm{d_t\vu^2-d_t\vu^1}^2_U\lesssim \norm{\dot{\vu}^2}^2_U+\norm{d_t\vu^1}^2_U.
\end{equation}
Note that $sd_t^2\vu^1=d_t\vu^1$ with the condition $d_t\vu^0=\bz$. 
Then by Lemma \ref{lem:norm-dtu-d2tu} (b), Proposition \ref{prop:init-step}, Proposition \ref{prop:trans-step} and the energy stability \eqref{eq:BDF2-stab-sum}, we derive 
\begin{equation}\label{eq:s2-d2tu}
		s^2\sum_{n=1}^{N}\norm{d_t^2\vu^n}^2_U\lesssim s^2\sum_{n=3}^{N}\norm{d_t\dot{\vu}^n}^2_U+s^2\norm{d_t^2\vu^2}^2_U+s^2\norm{d_t^2\vu^1}^2_U\leq C,
\end{equation}
with a generic constant $C>0$ independent of $s,N$. 
Substituting the discrete regularity \eqref{eq:s2-d2tu} into the constraint violation estimates \eqref{eq:bN-upper with initial velocity} yields \eqref{eq:bn-BDF2-orderO(s^2) estimate}.  
\end{proof}
}


\shuo{To conclude, Algorithm~\ref{Algorithm 2} extends naturally to higher-order BDF schemes by modifying $\dot{\vu}$ and $\hat{\vu}$ to their BDF-$k$ counterparts ($k \geq 3$). However, proving energy stability becomes challenging for these higher-order schemes due to the lack of $G$-stability. Consequently, only conditional estimates for constraint violations can be obtained, following the discussions in Section~\ref{sec:cons-vio-general} (see Corollary~\ref{cor:BDF3-cons-vio} and~\ref{cor:BDF4-cons-vio}). Unconditional estimates and convergence results analogous to Corollary~\ref{cor:limit-BDF2} remain unattainable for higher-order BDF versions of the Algorithm~\ref{Algorithm 2}.}

\setcounter{equation}{0}
\section{\shuo{Energy stable BDF-$k$ accelerated gradient methods}}\label{sec:modified} 
\shuo{This section presents a BDF-$k$ generalized ($k\geq2$) and modified version of Algorithm~\ref{Algorithm 2} in order to retain energy stability for higher-order BDF-$k$ schemes, namely when $k>2$. Sticking to the linearized constraint \eqref{eq:linear-cons}, the analysis in Section~\ref{sec:cons-vio-general} remains valid. The key adjustment adapts the third term in \eqref{eq:discrete_flow_BDF2} so that modified energy stability for arbitrary BDF-$k$ approximations is available.}

\shuo{This new form of energy stable accelerated gradient flow methods is as follows.
For an integer $k\geq1$, we compute $\dot{\vu}^n$ in the tangent space $\mathcal{F}_{\hat{\vu}^n}$, with definitions \eqref{eq:bdf-formula-extrapolation} and \eqref{eq:tangent_F}, through the following equation 
\begin{equation}\label{eq:discrete_flow_BDFk_modified}
	\innerprocuct{(\dot{\vu}^n-\dot\vu^{n-1})/s,\vw}_U+\innerprocuct{\frac{\alpha}{t_n}\dot{\vu}^n,\vw}_U+\mathcal{M} \left(\widetilde\vu^n,\vw\right)=0,
\end{equation}
for all $\vw\in\mathcal{F}_{\hat{\vu}^n}$, $t_n:=ns$, and $\widetilde\vu^n$ is defined as  
\begin{equation}\label{eq:def-tilde-un}
\widetilde\vu^n:=\sum_{j=0}^{k-1} \tilde\delta_j \vu^{n-j},
\end{equation}
with $\tilde\delta_j$ defined as in \eqref{eq:def-tilde-delta}. Compared to the scheme \eqref{eq:discrete_flow_BDF2}, \eqref{eq:discrete_flow_BDFk_modified} replaces $\vu^{n}$ in the third term by $\widetilde\vu^n$. Using the definition \eqref{eq:bdf-formula-extrapolation} for $\dot\vu^n$, it is straightforward to verify that
\begin{equation}\label{eq:udot-utilde}
s\dot\vu^n = \widetilde\vu^n - \widetilde\vu^{n-1}.
\end{equation}
With the choice $d_t\vu^0=\bz$, we formally have the convention that $\vu^{-1}:=\vu^0$, and therefore \eqref{eq:discrete_flow_BDFk_modified} defines for $n\geq k$. In practice, given $\vu^0\in\Ad$, we sequentially compute the initial values $\vu^1,\ldots,\vu^{k-1}$ using the corresponding BDF-$1$ to BDF-$(k-1)$ versions of \eqref{eq:discrete_flow_BDFk_modified}.
Furthermore, by substituting the expression $\delta_0\vu^n = s\dot\vu^n - \sum_{j=1}^k \delta_j \vu^{n-j}$ into the definition \eqref{eq:def-tilde-un} of $\widetilde\vu^n$ and utilizing the identities $\tilde\delta_j - \delta_j = \tilde\delta_{j-1}$ and $\delta_k = -\tilde\delta_{k-1}$, we obtain
\begin{equation}\label{eq:rewrite-tilde-un}
\widetilde\vu^n = s\dot\vu^n + \sum_{j=1}^{k-1} \tilde\delta_j \vu^{n-j}-\sum_{j=1}^{k} \delta_j \vu^{n-j}=s\dot\vu^n +  \sum_{j=1}^{k} \tilde\delta_{j-1} \vu^{n-j},
\end{equation}
in which $\dot\vu^n$ is the unknown for the equation in each iteration step. 
}

\shuo{
We summarize this generalized method in Algorithm \ref{Alg:modified}. 
\begin{algorithm}
		\leavevmode\newline(0)\textbf{Input}: $\vu^0\in\mathcal{A}$, $s>0$, $d_t\vu^0=\mathbf{0}$, $\alpha\ge3$, $\varepsilon>0$ and a $T_{max}>0$ that is sufficiently large.\\
        (1)\textbf{Initialization}: Compute $\vu^1,\ldots,\vu^{k-1}$ using the corresponding BDF-$1$ to BDF-$(k-1)$ versions of \eqref{eq:discrete_flow_BDFk_modified}. \\ \textbf{Set} $n=k$.\\   
		(2)\textbf{Compute} $\dot{\vu}^n\in\mathcal{F}_{\hat{\vu}^n}$ by the BDF-$k$ version of \eqref{eq:discrete_flow_BDFk_modified} with $\widetilde\vu^n$ given by \eqref{eq:rewrite-tilde-un}. \\
		(3)\textbf{Update} $\vu^n=s\delta_0^{-1}\dot\vu^n-\sum_{j=1}^{k}\delta_0^{-1}\delta_j \vu^{n-j}$. \\
		(4)\textbf{Stop when} $\seminorm{\frac{d_t}{2}\left.\left( \mathcal{M}(\widetilde\vu^n,\widetilde\vu^n)+\norm{\dot{\vu}^n}^2_U\right)\right.}\leq \varepsilon$ or $n>[T_{max}/s]$, \textbf{ otherwise set} $n\to n+1$ and \textbf{continue} with step (3). \\
        (5)\textbf{Output}: $\vu^n$.
        \caption{Energy stable BDF-$k$ accelerated gradient flow methods with tangent space update}\label{Alg:modified}
	\end{algorithm}
}    

\shuo{We next show that this method satisfies a modified energy stability.
\begin{thm}\label{thm:modified-stab} 
Let $\{\vu^n\}_{n\in\mathbb{N}}$ be a sequence of the outputs produced in Algorithm \ref{Alg:modified}. 
For every $n\geq k$, we have
    \begin{equation}\label{eq:modified-stab}
	\mathcal{M}(\widetilde\vu^n,\widetilde\vu^n)+\norm{\dot{\vu}^n}^2_U\leq \mathcal{M}(\widetilde\vu^{n-1},\widetilde\vu^{n-1})+\norm{\dot{\vu}^{n-1}}^2_U.
    \end{equation}
\end{thm}
\begin{proof} 
Choosing $\vw=\dot{\vu}^n$ in \eqref{eq:discrete_flow_BDFk_modified}, using the relation \eqref{eq:udot-utilde} and employing the binomial formula \eqref{binominal formula}, 
we obtain 
\begin{equation}\label{eq:modified-stab-onestep}
\begin{aligned}
	\left(\mathcal{M}(\widetilde\vu^n,\widetilde\vu^n)+\norm{\dot{\vu}^n}^2_U\right)+\frac{2\alpha}{n}\norm{\dot{\vu}^n}^2_U
	+s^2\norm{d_t\dot{\vu}^n}^2_U+s^2\M{\dot{\vu}^n,\dot{\vu}^n} =\mathcal{M}(\widetilde\vu^{n-1},\widetilde\vu^{n-1})+\norm{\dot{\vu}^{n-1}}^2_U,
\end{aligned}
\end{equation}
which immediately implies \eqref{eq:modified-stab}. 
\end{proof}
}

\shuo{
The energy stability immediately guarantees the vanishing velocity in the limit. 
\begin{cor}\label{cor:limit-BDFk-modified}
Let $\{\vu^n\}_{n\in\mathbb{N}}$ be a sequence of the outputs produced in Algorithm \ref{Alg:modified} and $s>0$ be fixed. As $n\to\infty$, $\dot\vu^n$ converges strongly to $\bz$ in $U$. For $2\leq k \leq 4$, there exists a $\vu^*\in U$ and a subsequence of $\vu^n$ such that $\vu^n\rightharpoonup\vu^*$ weakly in $U$.   
\end{cor}
\begin{proof}
Summing \eqref{eq:modified-stab-onestep} for $n$ and recalling the coercivity \eqref{assum:bilinear}, we derive
\begin{equation}\label{eq:dotu-sum-bound}
\sum_{n=k}^{\infty}\|\dot\vu^n\|_U^2\lesssim \mathcal{M}(\widetilde\vu^{k-1},\widetilde\vu^{k-1})+\norm{\dot{\vu}^{k-1}}^2_U. 
\end{equation}
Since $\vu^1,\ldots,\vu^{k-1}$ are generated by BDF-$1$ to BDF-$(k-1)$ versions of \eqref{eq:discrete_flow_BDFk_modified} respectively, we apply \eqref{eq:modified-stab} recursively for $n=k-1,\ldots,1$ with the BDF-$(k-1)$ to BDF-$1$ definitions of $\widetilde\vu^n$ and $\dot\vu^n$ respectively to conclude that $\mathcal{M}(\widetilde\vu^{k-1},\widetilde\vu^{k-1})+\norm{\dot{\vu}^{k-1}}^2_U\lesssim \mathcal{M}(\vu^{0},\vu^{0})$. Therefore, \eqref{eq:dotu-sum-bound} ensures that $\lim_{n\to\infty}\|\dot\vu^n\|_U=0$. 

By the definition \eqref{eq:def-tilde-un} and the relation \eqref{eq:sum-tilde-delta} of coefficients $\tilde\delta_j$, we derive 
\begin{equation}\label{eq:tildeun-rewrite}
\widetilde\vu^n=\vu^n-\vu^n+\sum_{j=0}^{k-1}\tilde\delta_j\vu^{n-j}=\vu^n+(\tilde\delta_0-1)\vu^n+\sum_{j=1}^{k-1}\tilde\delta_j\vu^{n-j}=\vu^n+s\sum_{\ell=0}^{k-2}(\sum_{m=0}^{\ell}\tilde\delta_m-1)d_t\vu^{n-\ell},
\end{equation}
which leads to 
\begin{equation}\label{eq:un-upper-tildeun}
\|\vu^n\|_U\lesssim \|\widetilde\vu^n\|_U+s\sum_{\ell=0}^{k-2}\|d_t\vu^{n-\ell}\|_U,
\end{equation}
where the hidden constant depends on $k$. 
We first note that $\|\widetilde\vu^n\|_U$ is uniformly bounded due to the modified stability \eqref{eq:modified-stab}. Towards estimating $\|d_t\vu^{n-\ell}\|_U$, we substitute \eqref{eq:def-tilde-un} into \eqref{eq:udot-utilde} to obtain 
\begin{equation}\label{eq:sdotu-rewrite}
\dot\vu^n=\sum_{j=0}^{k-1}\tilde\delta_j d_t\vu^{n-j}, 
\end{equation} 
and employ Lemma \ref{lem:diff-eq-sol-sharper-estimate} to derive that 
\begin{equation}\label{eq:dtun-sum-est}
\sum_{n=k}^{\infty}\|d_t\vu^n\|_U^2\lesssim \sum_{n=k}^{\infty}\|\dot\vu^n\|_U^2+\sum_{n=1}^{k-1}\|d_t\vu^n\|_U^2,
\end{equation}
which is uniformly bounded by summing the energy stability estimate \eqref{eq:modified-stab-onestep} and recalling the coercivity \eqref{assum:bilinear}. Therefore, $\|d_t\vu^{n}\|_U$ is uniformly bounded, and substituting it into \eqref{eq:un-upper-tildeun} yields $\|\vu^n\|_U\leq C$ with a generic constant $C$ depending on $k$ but not on $n$. Consequently, there exist a $\vu^*\in U$ and a subsequence of $\vu^n$ such that $\vu^n\rightharpoonup\vu^*$ weakly in $U$.  
\end{proof}
}

\shuo{
\begin{remark}
    The characterization of limits as in Corollary \ref{cor:convergence} extends directly to this modified algorithm \ref{Alg:modified} with $2\leq k\leq 4$ under the same assumptions. The primary technical difference involves establishing weak convergence $\widetilde{\vu}^n \rightharpoonup \vu^*$ and $\hat{\vu}^n \rightharpoonup \vu^*$ in $U$ (when $k=3,4$) for some subsequences.  
    
    In fact, by \eqref{eq:dtun-sum-est} and \eqref{eq:modified-stab-onestep} we observe that $d_t\vu^n\to\bz$ strongly in $U$. Then from \eqref{eq:tildeun-rewrite} and the fact that $\vu^n\rightharpoonup \vu^*$ up to a subsequence, we derive $\widetilde{\vu}^n \rightharpoonup\vu^*$ up to a subsequence as $n\to\infty$. 

    Moreover, we rewrite $\hat{\vu}^n$ using the definitions \eqref{eq:bdf-formula-extrapolation} and \eqref{eq:def-bdf-k-coeff} as 
    \begin{equation}\label{eq:uhatn-rewrite-k}
        \hat\vu^n=\vu^n-\vu^n+\hat\vu^n=\vu^n-s^kd^k_t\vu^n,
    \end{equation}
    with $k=3$ or $4$. Since $d^k_t\vu^n=s^{-1}(d^{k-1}_t\vu^n-d^{k-1}_t\vu^{n-1})$, we can rewrite $d^k_t\vu^n$ ($k=3$ or $4$) as linear combinations of $d_t\vu^{n-\ell}$ with $0\leq\ell\leq k-1$, which converge to $\bz$ strongly in $U$. Together with \eqref{eq:uhatn-rewrite-k}, this ensures that $\hat{\vu}^n \rightharpoonup \vu^*$ in $U$ up to a subsequence.    
\end{remark}
}

\shuo{
Conditional estimates for constraint violations of the arbitrary BDF-$k$ scheme can be established following the framework in Section \ref{sec:cons-vio-bdfn}, and in particular we refer to Corollary ~\ref{cor:BDF2-cons-vio}, ~\ref{cor:BDF3-cons-vio} and ~\ref{cor:BDF4-cons-vio} for examples of BDF-$2$ to BDF-$4$ schemes. 
We next show an unconditional estimate of constraint violations for Algorithm \ref{Alg:modified} with $2\leq k\leq 4$ in the following theorem.  
\begin{thm}\label{thm:BDFk_cons_vio_uncond}
    Let $\{\vu^n\}_{n\in\mathbb{N}}$ be a sequence of outputs produced by Algorithm \ref{Alg:modified}. For every integer $N\geq 1$, it satisfies that 
 \begin{equation}\label{eq:bn-BDFk-orderO(s^2) estimate}
     \norm{B(\vu^N,\vu^N)-\vu_c}_Z\lesssim s^2.
 \end{equation}
\end{thm}
\begin{proof}
Taking $d_t$ on both sides of \eqref{eq:sdotu-rewrite}, we obtain 
\begin{equation}\label{eq:sdtdotu}
d_t\dot\vu^n=\sum_{j=0}^{k-1}\tilde\delta_j d^2_t\vu^{n-j},
\end{equation}
which is a linear inhomogeneous difference equation for $d^2_t\vu^{n}$. We apply Lemma \ref{lem:diff-eq-sol-sharper-estimate} and energy stability \eqref{eq:modified-stab-onestep} to derive
\begin{equation}\label{eq:d2tun-sum-est}
s^2\sum_{n=k}^{N}\|d^2_t\vu^n\|_U^2\lesssim s^2\sum_{n=k}^{N}\|d_t\dot\vu^n\|_U^2+s^2\sum_{n=1}^{k-1}\|d^2_t\vu^n\|_U^2\leq C,
\end{equation}
with a generic constant $C>0$ independent of $s,N$. 
Substituting the discrete regularity \eqref{eq:d2tun-sum-est} into the constraint violation estimates \eqref{eq:bN-upper with initial velocity}, \eqref{eq:cons-upper-BDF3} yields \eqref{eq:bn-BDFk-orderO(s^2) estimate} for $k=2,3$. When $k=4$, we further employ the estimate $s^{2}\norm{d_t^3\vu^n}_U^2\leq 2\norm{d_t^2\vu^n}_U^2+2\norm{d_t^2\vu^{n-1}}_U^2$ and then the discrete regularity \eqref{eq:d2tun-sum-est} in \eqref{eq:cons-upper-BDF4} to conclude with \eqref{eq:bn-BDFk-orderO(s^2) estimate}.   
\end{proof}
}

\section{Numerical results}\label{sec:numerical results}
\setcounter{equation}{0}
In this section, we verify the theoretical results proved in previous sections by two benchmark examples in the set up of Example \ref{exampjle:anisotropic Dirichlet energy} and Example \ref{ex:plates}. 
\shuo{We reiterate that the computational acceleration effect of the proposed methods has been thoroughly investigated in \cite{DonGuoYan24}; our present focus is on quantifying constraint violations. First, we introduce some necessary notations.} 

The final outputs of Algorithms are denoted as $\vu_h^*$. The key quantity that we want to investigate is the constraint violation, which is defined as
\begin{equation}\label{eq:cons-vio-comp-2}
\delta^{cons}=\norm{B(\vu_h^*,\vu_h^*)-\shuo{\vu_c}}_{L^1(\Omega;\mathbb{R}^{n\times k})},
\end{equation}
\shuo{as $Z:=L^1(\Omega;\mathbb{R}^{n\times k})$ in these examples.}
The $L^1$ norm in \eqref{eq:cons-vio-comp-2} are practically evaluated with proper quadrature rules in each example, depending on the spatial discretization methods. 
We use different subscripts $GF$ and $AF$ to distinguish quantities for gradient flows or the proposed accelerated gradient flows for the purpose of comparison. 
We use $eoc$ to denote the experimental convergence order of the constraint violation. Moreover, the following quantities for accelerated gradient flows
\begin{equation}\label{eq:comp-quantity}
{\sigma^k=\sum_{n=3}^{N}\norm{d_t^k\vu^n_h}_U^2}\; \text{ and }\; \rho=\max\limits_{1\leq n\leq N}\norm{d_t^2\vu^n_h}_U^2,
\end{equation}
are calculated to reflect the discrete regularity assumptions for BDF-2/3/4 methods in the estimation of constraint violations, \shuo{where $N$ denotes the total number of iterations.} By computing \eqref{eq:comp-quantity}, we are able to validate if the regularity assumptions are satisfied or not by the considered examples.

\subsection{Anisotropic Dirichlet energy}
We start with an example of anisotropic Dirichlet energy.
	\begin{example}\label{example:anisotropic Dirichlet energy-comp}
 Consider the constrained minimization problem of anisotropic Dirichlet energy as described in Example \ref{exampjle:anisotropic Dirichlet energy}.
		For the spatial discretization we use the conforming finite element spaces of continuous elementwise affine functions.
		 
        In this example, we set $m=3,d=2,\Omega=(-1/2,1/2)^2,\Gamma^D=\partial \Omega$ and $\vu|_{\Gamma^D}=\vm|_{\partial \Omega}$, where 
        \begin{equation*}\label{eq:def-bc-dirichlet}
        \mathbf{m}(x)=(m_1(x),m_2(x),m_3(x)):=\frac{1}{1+\seminorm{x}^2}
		\begin{pmatrix}
			\sqrt{2}(x_1-x_2),\sqrt{2}(x_1+x_2),1-\seminorm{x}^2
		\end{pmatrix}^T, \; \text{ for } x\in \Omega \cup \partial \Omega.
        \end{equation*}
        We take an initial state $\vu_0\in\mathcal{A}$ where $\vu_0:=\frac{(m_1g_1,m_2g_2,m_3g_3)}{|(m_1g_1,m_2g_2,m_3g_3)|}$ with $\mathbf{g}(x)=(g_1(x),g_2(x),g_3(x))$ defined as 
        $$\mathbf{g}:=(1,1,1)-100(x_1-\frac12)(x_1+\frac12)(x_2-\frac12)(x_2+\frac12)(\sin (\frac12\pi x_1),8\sin (\frac12\pi x_2),16(x_1-x_2)\cos(8\pi (x_1+x_2)).$$
	In the tests, we choose $\vM=
		\begin{pmatrix}
			1&0\\
			0&10\\
		\end{pmatrix}$.
        A uniform triangulation $\mathcal{T}_h$ of $\Omega$ into $8192$ right-angled triangles is employed, and the discrete initial state $\vu^0_h$ is the nodal interpolation of $\vu_0$.
	\end{example}

We compare the computational performance, in particular the constraint violation in the  projection-free gradient flows and accelerated gradient flows with BDF-1 and BDF-2 scheme for this example and various time steps $s$. The \shuo{gradient flows are} discussed in \cite{bartels2016projection} and \cite{bartelsakrivis2023quadratic}. We set the stopping tolerance to be $\varepsilon=10^{-8}$ in the following simulations. Both the $L^2$-flow and $H^1$-flow are tested: in the former case, we take the flow metric to be $L^2$-inner product $(\cdot,\cdot)_{L^2}$ in the place of $(\cdot,\cdot)_U$ in equations like \eqref{eq:discrete_flow_BDF2}; for the latter case, $U=H^1(\Omega;\mathbb{R}^m)$. The results are reported in Table \ref{tab:BDF1_Anisotropic Dirichlet} and Table \ref{tab:BDF2_Anisotropic Dirichlet}. 
Moreover, we test the BDF-3 and BDF-4 methods for both projection-free gradient flows and \shuo{the energy stable accelerated gradient flows as in Algorithm \ref{Alg:modified}} with $H^1$ flow metric, and the results are presented in Table \ref{tab:BDF3_Anisotropic Dirichlet} and Table \ref{tab:BDF4_Anisotropic Dirichlet}.

\begin{center}
	\begin{longtable}{cccc cccc}
		\label{tab:BDF1_Anisotropic Dirichlet}
		$s$&  $\delta^{cons}_{GF}$& $eoc_{GF}$&$E_{GF}[\vu_h^*]$&$s$& $\delta_{AF}^{cons}$&$eoc_{AF}$&$E_{AF}[\vu_h^*]$\\
		\hline
		\hline
		\multicolumn{8}{c}{BDF-1 methods for $L^2$-gradient/accelerated gradient flow with $\alpha=25$}\\
		$2^{-10}$&1.118e-00&-&474.8&$2^{-5}$&4.870e-02&-&142.7\\
		$2^{-11}$&8.526e-01&0.39&405.1&$2^{-6}$&2.536e-02&0.94&60.26\\
		$2^{-12}$&6.145e-01&0.47&301.1&$2^{-7}$&1.311e-03&0.95&30.82\\
		$2^{-13}$&4.080e-01&0.59&199.3&$2^{-8}$&6.499e-03&1.01&21.01\\
		$2^{-14}$&2.539e-01&0.68&120.7&$2^{-9}$&3.267e-03&0.99&17.92\\
		\hline	
		\multicolumn{8}{c}{BDF-1 methods for $H^1$-gradient/accelerated gradient flow with $\alpha=25$}\\
		$2^{-1}$&1.510e-00& - &258.0&$2^{-1}$&4.004e-01&-&42.27\\
		$2^{-2}$&9.728e-01&0.63&179.2&$2^{-2}$&1.999e-01&1.00&25.07\\
		$2^{-3}$&6.081e-01&0.68&92.68&$2^{-3}$&9.917e-02&1.01&17.44\\
		$2^{-4}$&3.746e-01&0.70&47.72&$2^{-4}$&4.934e-02&1.01&17.55\\
		$2^{-5}$&1.982e-01&0.92&27.55&$2^{-5}$&2.460e-02&1.00&16.82\\
		\hline      
		\caption{Example \ref{example:anisotropic Dirichlet energy-comp}: comparison between gradient flows and \shuo{the accelerated gradient flow (Algorithm \ref{Algorithm 1})} with BDF-1 scheme.}
	\end{longtable}
\end{center}

We summarize the observations from the computational results in Tables \ref{tab:BDF1_Anisotropic Dirichlet}--\ref{tab:BDF4_Anisotropic Dirichlet} of these projection-free methods as follows. 

    For gradient flows, the constraint violation is asymptotically $\order{s}$ when using the BDF-1 discretization and not better than $\order{s^2}$ when using the BDF-2, BDF-3 and BDF-4 schemes. For accelerated gradient flows, with BDF-1 method the constraint violation achieves $\order{s}$, with BDF-2 and BDF-3 schemes it is asymptotically $\order{s^3}$, and it is $\order{s^4}$ when using the BDF-4 discretization. The observed orders are consistent with the unconditional/conditional estimates of constraint violations that we have proved in previous sections for BDF-1, BDF-2, BDF-3 and BDF-4 methods. 

\begin{center}
    \begin{longtable}{ccccc ccccc}
		\label{tab:BDF2_Anisotropic Dirichlet}
		$s$&  $\delta^{cons}_{GF}$& $eoc_{GF}$&$s^2\sigma^2_{GF}$&$E_{GF}[\vu_h^*]$&$s$& $\delta_{AF}^{cons}$&$eoc_{AF}$&$s\sigma^2_{AF}$&$E_{AF}[\vu_h^*]$\\
		\hline
		\hline
		\multicolumn{10}{c}{BDF-2 methods for $L^2$-gradient/accelerated gradient flow with $\alpha=25$}\\
		$2^{-10}$&1.974e-00&-&1.272e+06&489.7&$2^{-5}$&2.324e-01&-&4.174e+03&144.5\\
		$2^{-11}$&1.136e-00&0.58&3.348e+06&473.4&$2^{-6}$&5.769e-02&2.01&7.309e+03&35.91\\
		$2^{-12}$&7.941e-01&0.73&7.954e+06&438.3&$2^{-7}$&1.231e-02&2.23&1.119e+04&17.95\\
		$2^{-13}$&4.324e-01&0.88&1.685e+07&358.4&$2^{-8}$&2.150e-03&2.52&1.470e+04&16.45\\
		$2^{-14}$&2.129e-01&1.02&3.144e+07&206.3&$2^{-9}$&3.191e-04&2.75&1.695e+04&16.37\\
		$2^{-15}$&9.697e-02&1.13&5.300e+07&102.5&$2^{-10}$&4.298e-05&2.89&1.805e+04&16.36\\
            $2^{-16}$&4.140e-02&1.23&8.312e+07&40.23&$2^{-11}$&5.554e-06&2.95&1.856e+04&16.36\\
		\hline	
		\multicolumn{10}{c}{BDF-2 methods for $H^1$-gradient/accelerated gradient flow with $\alpha=25$}\\
		$2^{-0}$&4.493e-00&-&2.070e+03&407.8&$2^{-0}$&2.204e-01&-&112.6&28.04\\%
		$2^{-1}$&2.779e-00&0.69&5.332e+03&302.0&$2^{-1}$&3.649e-02&2.47&141.5&17.26\\
		$2^{-2}$&1.359e-00&1.03&1.143e+04&205.6&$2^{-2}$&5.688e-03&2.68&162.2&16.46\\
		$2^{-3}$&5.591e-01&1.28&2.018e+04&92.64&$2^{-3}$&7.931e-04&2.84&172.3&16.37\\
		$2^{-4}$&2.082e-01&1.43&2.920e+04&32.01&$2^{-4}$&1.039e-04&2.93&176.4&16.36\\
		$2^{-5}$&6.472e-02&1.69&3.544e+04&18.77&$2^{-5}$&1.325e-05&2.97&178.1&16.36\\
		$2^{-6}$&1.845e-02&1.81&3.851e+04&16.76&$2^{-6}$&1.670e-06&2.98&178.9&16.36\\
		\hline    		
    \caption{Example \ref{example:anisotropic Dirichlet energy-comp}: comparison between gradient flow and \shuo{the accelerated gradient flow (Algorithm \ref{Algorithm 2})} with BDF-2 scheme.}
    \end{longtable}
\end{center}

\begin{center}
\begin{longtable}{ccccc ccccc}
\label{tab:BDF3_Anisotropic Dirichlet}
$s$&  $\delta^{cons}_{GF}$& $eoc_{GF}$&$s^2\sigma^2_{GF}$&$E_{GF}[\vu^*_h]$&
$s$&  $\delta^{cons}_{AF}$& $eoc_{AF}$&$s\sigma^2_{AF}$&$E_{AF}[\vu^*_h]$
\\
\hline
\hline
\multicolumn{10}{c}{BDF-3 method for $H^1$-gradient/\shuo{accelerated gradient flow} with $\alpha=25$}\\
$2^{-1}$&1.260e-00&-&4.358e+03&130.4&$2^{-1}$&1.295e-02&-&121.7&16.24\\
$2^{-2}$&5.663e-01&1.15&8.594e+03&87.80&$2^{-2}$&3.752e-03&1.79&147.3&16.30\\
$2^{-3}$&2.361e-01&1.26&1.502e+04&42.21&$2^{-3}$&6.628e-04&2.50&162.9&16.35\\
$2^{-4}$&8.614e-02&1.45&2.193e+04&21.80&$2^{-4}$&9.581e-05&2.79&171.1&16.36\\
$2^{-5}$&2.819e-02&1.61&2.748e+04&17.32&$2^{-5}$&1.274e-05&2.91&175.4&16.36\\
$2^{-6}$&8.673e-03&1.70&3.114e+04&16.56&$2^{-6}$&1.639e-06&2.96&177.5&16.36\\
\hline
\caption{{Example \ref{example:anisotropic Dirichlet energy-comp}: comparison between gradient flow and \shuo{the accelerated gradient flow (Algoritm \ref{Alg:modified})} with BDF-3 scheme. }}
\end{longtable}
\end{center}

\begin{center}
	\begin{longtable}{ccccc ccccc}
		\label{tab:BDF4_Anisotropic Dirichlet}
		$s$&  $\delta^{cons}_{GF}$& $eoc_{GF}$&$s^2\rho_{GF}$&$s^4\sigma^3_{GF}$&
		$s$&  $\delta^{cons}_{AF}$& $eoc_{AF}$&$\rho_{AF}$&$s^2\sigma^3_{AF}$\\
		\hline
		\hline
		\multicolumn{10}{c}{BDF-4 method for $H^1$-gradient/\shuo{ accelerated gradient flow} with $\alpha=25$}\\
		$2^{-1}$&1.110e-00&-   &2.772e+03&2.558e+03&$2^{-1}$&7.765e-03&-&45.02&37.55\\
		$2^{-2}$&6.757e-01&0.72&6.366e+03&3.044e+03&$2^{-2}$&4.273e-04&4.18&51.87&35.88\\
		$2^{-3}$&2.774e-01&1.28&1.195e+04&3.211e+03&$2^{-3}$&3.718e-05&3.52&60.46&25.74\\
		$2^{-4}$&1.067e-01&1.38&1.846e+04&2.817e+03&$2^{-4}$&1.756e-06&4.40&63.02&16.00\\
		$2^{-5}$&3.400e-02&1.65&2.428e+04&1.912e+03&$2^{-5}$&7.712e-08&4.51&64.01&9.382\\
		$2^{-6}$&9.842e-03&1.79&2.848e+04&1.095e+03&$2^{-6}$&4.372e-09&4.14&64.49&6.039\\
		\hline	
		\caption{{Example \ref{example:anisotropic Dirichlet energy-comp}: comparison between gradient flow and \shuo{the accelerated gradient flow (Algoritm \ref{Alg:modified})} with BDF-4 scheme. }}
	\end{longtable}
\end{center}

    The accelerated gradient flow (AF) outperforms the gradient flow (GF) in terms of constraint violations. From Tables \ref{tab:BDF1_Anisotropic Dirichlet} and \ref{tab:BDF2_Anisotropic Dirichlet},
    GF requires a much smaller $s$ in order to observe an asymptotic linear/quadratic consistency in constraints for BDF-1/2, respectively, while AF reaches either the same or better rates with much looser step sizes.
   Moreover, AF provides higher order convergence compared to GF with BDF-2/3/4 schemes. 
    
    In all the cases, the $H^1$-flow metric provides a better computational performance than the $L^2$-flow metric, as the latter requires a much smaller $s$ to reach the constraint violation at the same level, compared to the former. 
    
    In the tests with BDF-2/3 methods, it is clear that $s\sigma^2_{AF}$ is uniformly bounded so that the discrete regularity assumption in our theory for accelerated gradient flows is satisfied. Moreover, in Table \ref{tab:BDF4_Anisotropic Dirichlet}, $\rho_{AF}$ and $s^2\sigma^3_{AF}$ are also uniformly bounded, which implies that the discrete regularity assumption \eqref{eq:discrete-regularity-BDF4} also holds for this example to attain $\order{s^4}$ convergence for the \shuo{BDF-4 version of Algorithm \ref{Alg:modified}}. 
    
    Recall that the method developed for estimating the constraint violation, as outlined in Section \ref{sec:cons-vio-bdfn}, \shuo{exploits only the algebraic structure of quadratic constraints in \eqref{eq:admissible_set} and the linearized constraint \eqref{eq:linear-cons} with BDF approximations and extrapolations; equations from iterative methods are only required} in the final step to establish discrete regularity and determine constraint consistency orders. The columns $s^2\sigma^2_{GF}$, $s^2\rho_{GF}$ and $s^4\sigma^3_{GF}$ for gradient flows with BDF-2/3/4 schemes reflect the practical discrete regularity ensured by the original gradient flow for this example, thereby justifying the reduced order $eoc_{GF}$ for these schemes compared to accelerated gradient flows. 

In this example, we infer that the exact solution may exhibit a strong fit to the linear FEM space, as we observe no obstruction to the predicted convergence rates when fixing the spatial refinement while decreasing the temporal step size. Additional tests (not reported for brevity) indicate that highly accurate solutions can be achieved with even coarser meshes; specifically, the constraint violation $\delta^{cons}_{AF}$ can reach $10^{-14}$ with a mesh of $512$ elements and a very small time step. Time refinement is particularly dominant in this case. However, as illustrated in the subsequent example, more complex problems may exhibit different phenomena.


\subsection{Prestrained plates}
We next test the proposed methods with the model of prestrained plates. In particular, with this example, we see the influence of spatial discretization refinements on the constraint violations.
\begin{example}[Prestrained plates with clamped boundary]\label{ex:plates-comp-2}
We consider the prestrained plates model as described in Example \ref{ex:plates}.
The domain is taken as $\Omega:=(-5,5)\times(-2,2)\subset\mathbb{R}^2$, and we impose clamped boundary condition on the side $\Gamma^D:=\{x_1=-5\}$:
\begin{equation}\label{eq:clamped-BC}
\phi(x_1,x_2)=(x_1,x_2,0)^T \quad \mbox{and} \quad \psi=[I_2,\mathbf{0}]^T \qquad (x_1,x_2) \in \Gamma^D.
\end{equation}
We take metric $A$ to be 
\begin{equation}\label{eq:metric-ex3}
A(x_1,x_2)=\begin{bmatrix}
1+c^2(2(x_1+5)(x_1-2)+(x_1+5)^2)^2       & 0 \\
0       & 1
\end{bmatrix},
\end{equation}
with the parameter $\shuo{c=10^{-2}}$. 
The initial state is $\vy^0(x_1,x_2)=(x_1,x_2,c(x_1+5)^2(x_1-2))$ satisfying $\I[\vy^0]=A$ and the boundary condition \eqref{eq:clamped-BC}. 
We use Morley finite element method \cite{morley1968triangular,brenner2013morley,DonGuoYan24} for spatial discretization of this model.
\end{example}

We set the stopping tolerance parameter $\varepsilon=10^{-6}$ in this example and take $\shuo{\alpha=16}$ in Algorithm \ref{Algorithm 2}. The results are reported in Table \ref{tab:plates}.


\setcounter{table}{3}
\begin{table}[ht]
\begin{minipage}{0.42\textwidth}
	\begin{longtable}{ccccc}
		$s$&$\delta^{cons}_{AF}$& $eoc_{AF}$ & $s\sigma^2_{AF}$ &$E_{AF}$\\
		\hline
		\hline
		\multicolumn{5}{c}{$512$ elements}\\
        $3.2$ & 1.894e-00 & - & 7.329e-03 &2.243e-01 \\
        $1.6$ & 3.646e-01 & 2.38 & 1.046e-02 &2.117e-01 \\
		$0.8$ & 6.857e-02 & 2.41 & 1.244e-02 &2.100e-01 \\
		$0.4$ & 2.284e-02 & 1.59 & 1.388e-02 &2.098e-01 \\
		$0.2$ & 1.966e-02 & 0.22 & 1.460e-02 &2.098e-01 \\
		$0.1$ & 1.959e-02 & 0.01 & 1.492e-02 &2.098e-01 \\
		\hline	
		\multicolumn{5}{c}{$2048$ elements}\\
        $3.2$ & 1.882e-00 & - & 7.332e-03 &2.234e-01 \\
        $1.6$ & 3.535e-01 & 2.41 & 1.046e-02 &2.108e-01  \\
        $0.8$ & 5.757e-02 & 2.62 & 1.245e-02 &2.091e-01 \\
		$0.4$ & 1.120e-02 & 2.36 & 1.389e-02 &2.089e-01 \\
		$0.2$ & 5.088e-03 & 1.14 & 1.462e-02 &2.088e-01 \\
		$0.1$ & 4.908e-03 & 0.05 & 1.494e-02 &2.088e-01 \\
		\hline	    
	\end{longtable}
 \end{minipage}
 \hspace{0.08\textwidth}
 \begin{minipage}{0.42\textwidth}
	\begin{longtable}{ccccc}
		$s$&$\delta^{cons}_{AF}$& $eoc_{AF}$ & $s\sigma^2_{AF}$ &$E_{AF}$\\
		\hline
		\hline
		\multicolumn{5}{c}{$8192$ elements}\\
        $1.6$ & 3.507e-01 &-& 1.046e-02 &2.106e-01 \\
        $0.8$ & 5.484e-02 & 2.68 & 1.246e-02 &2.089e-01 \\
		$0.4$ & 8.445e-03 & 2.70 & 1.389e-02 &2.086e-01 \\
		$0.2$ & 1.913e-03 & 2.14 & 1.462e-02 &2.086e-01 \\
		$0.1$ & 1.235e-03 & 0.63 & 1.494e-02 &2.086e-01 \\
		$0.05$ & 1.227e-03 & 0.01 & 1.508e-02 &2.086e-01 \\
		\hline	
		\multicolumn{5}{c}{$32768$ elements}\\
        $1.6$ & 3.501e-01 &-& 1.047e-02 &2.105e-01 \\
        $0.8$ & 5.416e-02 & 2.69 & 1.246e-02 &2.088e-01 \\
		$0.4$ & 7.761e-03 & 2.80 & 1.390e-02 &2.086e-01 \\
		$0.2$ & 1.216e-03 & 2.67 & 1.462e-02 &2.086e-01 \\
		$0.1$ & 3.636e-04 & 1.74 & 1.494e-02 &2.086e-01 \\
		$0.05$ & 3.071e-04 & 0.24 & 1.508e-02 &2.086e-01 \\
		\hline	    
	\end{longtable}
 \end{minipage}
 \caption{Example \ref{ex:plates-comp-2}: accelerated gradient flows with BDF-2 scheme \shuo{(Algorithm \ref{Algorithm 2})} for various spatial discretization refinements.}
 \label{tab:plates}
\end{table}
\begin{table}[ht]
    \centering
    \begin{tabular}{ccccc}
      Method & $s$ & $N$ & $\delta^{\text{cons}}$ & $E$\\ \hline
      GF-BDF1 & $1/1280$ & $45943$ & 2.071e-02 & 2.0980e-01 \\ 
      GF-BDF2 & $1/40$ & $617$ & 2.068e-02  & 2.0977e-01\\
      AF-BDF1 & $1/1280$ & $23296$ & 2.136e-02 & 2.0981e-01 \\
      AF-BDF2 & $1/5$ & $159$ & 1.966e-02 & 2.0976e-01\\ \hline
    \end{tabular}
    \caption{\shuo{Example \ref{ex:plates-comp-2}: comparison between different methods. ``GF-BDF1'' stands for the classical backward Euler projection-free gradient flow methods widely used in the literature, ``GF-BDF2'' the BDF-2 projection-free gradient flows investigated in \cite{bartelsakrivis2023quadratic}, ``AF-BDF1'' the Algorithm \ref{Algorithm 1} with $\alpha=3$ and ``AF-BDF2'' the Algorithm \ref{Algorithm 2} with $\alpha=16$.}}
    \label{tab:limit-cons}
\end{table}
   
From Table \ref{tab:plates}, we observe that for each fixed mesh refinement, the empirical order of convergence ($eoc$) decreases as the time step $s$ reducing to a certain level. This behavior indicates that the constraint violation $\delta^{cons}$ stabilizes and approaches its limit depending on the current spatial discretization size.
{Taking into account the computational load of this example, the spatial discretization is refined in 4 levels with the finest to be 32768 elements in our tests.
We can see clearly that the convergence rate of the constraint violations gets improved when finer mesh is applied.
Using the finest mesh, we achieve a rate of {nearly} $\order{s^3}$ prior to approaching the limitations imposed by the mesh size.} In contrast, for coarser meshes, such as those with fewer than $32768$ elements, the algorithm produces slightly smaller rates.

Examining the stabilized $\delta^{cons}$ ({corresponding to the smallest $s$}) for each spatial refinement, we observe that $\delta^{cons}$ behaves as $\order{h^2}$, where $h$ denotes the mesh size. Furthermore, it is evident that $s\sigma^2$ remains uniformly bounded across all cases, thereby satisfying the discrete regularity assumption stated in \eqref{eq:discrete regularity}. Overall, this test for the Example \ref{ex:plates-comp-2} demonstrates that, for complex problems like the prestrained plates model, spatial discretization significantly influences the computational efficiency of the proposed algorithms. 

\shuo{
Furthermore, Table~\ref{tab:limit-cons} compares the time step size $s$ and iteration number $N$ required to reach the limiting constraint violation $\delta^{\text{cons}} \approx 0.02$ for various methods in Example~\ref{ex:plates-comp-2} with a $512$-element mesh and $\varepsilon = 10^{-6}$. 

We aim to compare the efficiency of these projection-free methods when they reach a comparable accuracy.  
The results demonstrate that: (1) BDF2 methods significantly outperform BDF1, and (2) accelerated flows are more efficient than gradient flows in each case. Notably, even for the same time step $s$, accelerated flows usually converge faster than the classical gradient flows for this example. In fact, this acceleration effect can be further enhanced through backtracking strategies, and we refer to \cite{DonGuoYan24} for a more detailed computational investigation.

}

\section{Conclusion}
\shuo{This paper has proposed and analyzed a family of novel algorithms that integrate strategies from multiple fields to solve a class of non-convex energy minimization problems, namely the quadratic energy minimization with quadratic equality constraints. We developed a general framework for estimating constraint violations in arbitrary BDF-$k$ projection-free iterative schemes.
When applied to accelerated gradient flows, this framework has demonstrated that the accelerated flows achieve improved error bounds in constraint violation compared to existing projection-free gradient flows. 
Specifically, we have established unconditional/conditional constraint violation estimates, modified energy stability, and weak convergence for BDF-$2,3,4$ schemes of the accelerated flows. 
Numerical simulations have confirmed these theoretical results and shown that our methods can achieve higher accuracy with reduced computational cost.
  
We note that several theoretical aspects of the optimization problem remain open, including the qualitative and quantitative error estimates of the proposed flows' trajectories to solutions of the non-convex variational problems. Furthermore, the formally introduced continuous flows constitute a class of hyperbolic partial differential equations that may inspire additional theoretical and numerical investigations.
}


\subsection*{Acknowledgement}
\shuo{
The authors thank the anonymous reviewers for the detailed and insightful comments which helped to improve the paper. 
The authors also thank Hansen Li for the discussions inspiring the energy stable BDF-k accelerated schemes.
}
The work of G. Dong and Z. Gong was supported by the NSFC grants No. 12001194 and No. 12471402. The work of Z. Gong and Z. Xie was supported by the Key Project of Xiangjiang Laboratory 22XJ01013 and the NSFC grant No. 12171148. The work of S. Yang was supported by the NSFC grant No.12401512.  

\appendix
\section{Proof of Lemma \ref{lem:alg-identity}}\label{appendix:appendix-A}
\shuo{
\begin{proof}
We need to compare coefficients of $a_{n-m}a_{n-p}$ for all $0\le m\le p\le k$ on both sides of \eqref{eq:alg-identity} to verify the existence of $\beta_{j\ell}$. It is clear that this can be formulated as a system of $(k+2)(k+1)/2$ linear equations for $(k+1)k/2$ unknowns $\beta_{j\ell}$: 
\begin{equation}\label{eq:linear-system-identity}
A\vx=\vr,
\end{equation} 
where $\vx$ is the $(k+1)k/2\times1$ vector made of $\beta_{j\ell}$'s, $\vr$ represents the $(k+2)(k+1)/2\times1$ vector made of coefficients of $a_{n-m}a_{n-p}$ ($0\le m\le p\le k$) on the LHS of \eqref{eq:alg-identity}, and $A$ is the $(k+2)(k+1)/2\times(k+1)k/2$ matrix so that $A\vx$ represents the same coefficients on the RHS of \eqref{eq:alg-identity}. 

Towards the existence of $\beta_{j\ell}$, it suffices to show that there are $k+1$ rows in this linear system are redundant.    
More precisely, we use the index $(m,p)$ with $0\le m\le p\le k$ to denote the row of linear system \eqref{eq:linear-system-identity} that corresponds to the coefficients of $a_{n-m}a_{n-p}$. In fact, in step (i)-(iii) we show 
\begin{equation}\label{eq:sys-rel-rhs}
-2\vr_{m,m}=\sum_{p<m}\vr_{p,m}+\sum_{p=m+1}^{k}\vr_{m,p},
\end{equation}
\begin{equation}\label{eq:sys-rel-lhs}
-2(A\vx)_{m,m}=\sum_{p<m}(A\vx)_{p,m}+\sum_{p=m+1}^{k}(A\vx)_{m,p},
\end{equation}
for any $0\le m\le k$.

We remove these redundant rows with indices $(m,m)$ for $0\le m\le k$ in $A$, and denote the new matrix of size $(k+1)k/2\times(k+1)k/2$ as $\tilde A$. In step (iv), we prove that $\tilde A$ is a lower triangular matrix with a proper permutation and thus invertible, which further guarantees the uniqueness of solution $\beta_{j\ell}$.   

{\bf Step (i), preliminary calculations.} By \eqref{eq:def-bdf-k-coeff} and a direct calculation, we conclude that 
\begin{align}\label{eq:LHS-alg}
\text{LHS of \eqref{eq:alg-identity}}&=\delta_0a_n^2+\sum_{m=1}^k2(-1)^m\delta_0\binom{k}{m}a_na_{n-m}+\sum_{m=1}^k\left(2(-1)^m\delta_m\binom{k}{m}+\delta_m\right)a_{n-m}^2 \\ \nonumber
&+\sum_{m=1}^{k-1}\sum_{p=m+1}^k\left(2(-1)^p\delta_m\binom{k}{p}+2(-1)^m\delta_p\binom{k}{m}\right)a_{n-m}a_{n-p}. 
\end{align}
For the RHS of \eqref{eq:alg-identity}, employing \eqref{eq:Backward-formula-expression}, we derive that 
\begin{equation}\label{eq:RHS-alg-1}
    s^{2j}(d^j_ta_{n-\ell})^2=\sum_{m=0}^j\binom{j}{m}^2a_{n-\ell-m}^2+\sum_{m=0}^{j-1}\sum_{p=m+1}^j2(-1)^{m+p}\binom{j}{m}\binom{j}{p}a_{n-\ell-m}a_{n-\ell-p}. 
\end{equation}

{\bf Step (ii), proof of \eqref{eq:sys-rel-rhs}.} When $m=0$, using \eqref{eq:LHS-alg} and \eqref{eq:useful-relation}, it is clear that 
\begin{equation*}
-2\vr_{0,0}=-2\delta_0=2\delta_0\sum_{p=1}^k(-1)^p\binom{k}{p}=\sum_{p=1}^{k}\vr_{0,p}.
\end{equation*}
When $m\ge1$, by \eqref{eq:LHS-alg} and \eqref{eq:useful-relation}, we derive 
\begin{align*}
\frac12\left(\sum_{p<m}\vr_{p,m}+\sum_{p=m+1}^{k}\vr_{m,p}\right)&=\delta_0(-1)^m\binom{k}{m}+\sum_{p=1,p\ne m}^k\left((-1)^p\delta_m\binom{k}{p}+(-1)^m\delta_p\binom{k}{m}\right) \\
&=(-1)^m\binom{k}{m}\sum_{p=0,p\ne m}^k\delta_p+\left(\sum_{p=1,p\ne m}^k(-1)^p\binom{k}{p}\right)\delta_m \\
&=-2(-1)^m\binom{k}{m}\delta_m-\delta_m=-\vr_{m,m}.
\end{align*}

{\bf Step (iii), proof of \eqref{eq:sys-rel-lhs}.} From \eqref{eq:RHS-alg-1}, we derive 
\begin{equation*}
(A\vx)_{m,m}=\sum_{\ell=0}^m\sum_{j=m-\ell,j\ne0}^{k-\ell}\beta_{j\ell}\binom{j}{m-\ell}^2,
\end{equation*}
and
\begin{align*}
\sum_{p<m}(A\vx)_{p,m}+\sum_{p=m+1}^{k}(A\vx)_{m,p}&=2\left(\sum_{p<m}\sum_{\ell=0}^p\sum_{j=m-\ell}^{k-\ell}+\sum_{p=m+1}^k\sum_{\ell=0}^m\sum_{j=p-\ell}^{k-\ell}\right)\beta_{j\ell}(-1)^{p+m-2\ell}\binom{j}{m-\ell}\binom{j}{p-\ell} \\ 
&=2\left(\sum_{\ell=0}^{m-1}\sum_{j=m-\ell}^{k-\ell}\sum_{p=\ell}^{m-1}+\sum_{\ell=0}^m\sum_{j=m+1-\ell}^{k-\ell}\sum_{p=m+1}^{j+\ell}\right)\beta_{j\ell}(-1)^{p+m-2\ell}\binom{j}{m-\ell}\binom{j}{p-\ell}.
\end{align*}
Letting $q=p-\ell$ and separating the case $j=m-\ell$ and $\ell=m$ from the first and second term in the above expression, we further reach 
\begin{align*}
\frac12\left(\sum_{p<m}(A\vx)_{p,m}+\sum_{p=m+1}^{k}(A\vx)_{m,p}\right)&=\sum_{\ell=0}^{m-1}\sum_{j=m+1-\ell}^{k-\ell}\beta_{j\ell}\binom{j}{m-\ell}(-1)^{m-\ell}\sum_{q=0,q\ne m-\ell}^j(-1)^q\binom{j}{q} \\
&+\sum_{\ell=0}^{m-1}\beta_{m-\ell,\ell}(-1)^{m-\ell}\sum_{q=0}^{m-\ell-1}(-1)^q\binom{m-\ell}{q}\\
&+\sum_{j=1}^{k-m}\beta_{jm}\sum_{q=1}^{j}(-1)^q\binom{j}{q}.
\end{align*}
Moreover, we apply \eqref{eq:useful-relation} again to obtain 
\begin{equation*}
\sum_{q=0,q\ne m-\ell}^j(-1)^q\binom{j}{q}=-(-1)^{m-\ell}\binom{j}{m-\ell},\quad \sum_{q=0}^{m-\ell-1}(-1)^q\binom{m-\ell}{q}=-(-1)^{m-\ell},\quad\sum_{q=1}^{j}(-1)^q\binom{j}{q}=-1,
\end{equation*}
and incorporating the fact that $\binom{j}{m-\ell}^2=1$ when $j=m-\ell$ or $m=\ell$ to conclude that 
\begin{align*}
\frac12\left(\sum_{p<m}(A\vx)_{p,m}+\sum_{p=m+1}^{k}(A\vx)_{m,p}\right)&=-\sum_{\ell=0}^m\sum_{j=m-\ell,j\ne0}^{k-\ell}\beta_{j\ell}\binom{j}{m-\ell}^2=-(A\vx)_{m,m}.
\end{align*}

{\bf Step (iv), uniqueness of solution.} 
For $0\le m<p\le k$, 
\begin{equation}\label{eq:Ax-mp}
(A\vx)_{m,p}=2\sum_{\ell=0}^m\sum_{j=p-\ell}^{k-\ell}\beta_{j\ell}(-1)^{p+m-2\ell}\binom{j}{m-\ell}\binom{j}{p-\ell}, 
\end{equation}
which means that the smallest index $j$ so that $\beta_{j\ell}$ has non-zero coefficient on the RHS of \eqref{eq:Ax-mp} is $j=p-m$, and there is only one such term, i.e., $\beta_{p-m,m}$. We first remove $k+1$ rows of $A$ with indices $(m,m)$ for $0\le m\le k$, and denote it as $\tilde A$. 
We next rearrange rows of $\tilde A$ by permuting indices so that $p-m=k,\ldots,1$, namely we put the rows in the following order of indices for $(m,p)$:
\begin{equation}\label{eq:indices-order}
(0,k)\to (0,k-1) \to (1,k) \to (0,k-2) \to (1,k-1) \to (2,k) \to \cdots, 
\end{equation}
and we rearrange columns of $\tilde A$ and $\vx$ so that $\vx$ is made of element $\beta_{p-m,m}$, where $(m,p)$ is taken in the order of \eqref{eq:indices-order}.

By \eqref{eq:Ax-mp}, this process guarantees that $\tilde A$ is a lower triangular matrix.  
Therefore, the solution to the linear system \eqref{eq:linear-system-identity} is unique, and so is the $\beta_{j\ell}$ that satisfy \eqref{eq:alg-identity}.      
\end{proof}
}

\bibliographystyle{unsrt}
\bibliography{ref}
\end{document}